\documentclass[11pt,oneside, letter]{article}

\usepackage{amsthm, amsmath, amssymb, amsfonts, graphicx, epsfig}
\usepackage{accents}

\usepackage{bbm}
\usepackage{mathrsfs}

\usepackage{ulem}

\usepackage{cancel}

\usepackage{epstopdf}

\usepackage{graphicx}
\usepackage[font=sl,labelfont=bf]{caption}
\usepackage{subcaption}

\usepackage{float}
\restylefloat{table}

\usepackage{enumerate}
\usepackage[colorlinks=true, pdfstartview=FitV, linkcolor=blue,
            citecolor=blue, urlcolor=blue]{hyperref}
\usepackage[usenames]{color}

\usepackage{url}
\makeatletter\def\url@leostyle{%
 \@ifundefined{selectfont}{\def\UrlFont{\sf}}{\def\UrlFont{\scriptsize\ttfamily}}} \makeatother

\usepackage[margin=1.0in, letterpaper]{geometry}

\newtheorem{theorem}{Theorem}
\newtheorem{proposition}[theorem]{Proposition}
\newtheorem{lemma}[theorem]{Lemma}

\theoremstyle{definition}

\newtheorem{example}[theorem]{Example}
\theoremstyle{remark}
\newtheorem{remark}[theorem]{Remark}

\numberwithin{equation}{section}
\numberwithin{theorem}{section}

\definecolor{Red}{rgb}{0.9,0,0.0}
\definecolor{Blue}{rgb}{0,0.0,1.0}

\def\cB{\mathcal{B}}
\def\cC{\mathcal{C}}

\def\cF{\mathcal{F}}

\def\bE{\mathbb{E}}
\def\bF{\mathbb{F}}

\def\bH{\mathbb{H}}

\def\bN{\mathbb{N}}

\def\bP{\mathbb{P}}

\def\bR{\mathbb{R}}

\def\bT{\mathbb{T}}

\def\bV{\mathbb{V}}

\newcommand{\1}{\mathbbm{1}}            
\newcommand{\set}[1]{\{#1\}}            
\renewcommand{\mid}{\;|\;}              
\renewcommand{\d}{\operatorname{d}\!}   

\DeclareMathOperator*{\argmax}{arg\,max} 

\makeatletter
\def\namedlabel#1#2{\begingroup
    #2%
    \def\@currentlabel{#2}%
    \phantomsection\label{#1}\endgroup
}
\makeatother

\title{ 
Long-run risk sensitive dyadic impulse control}

\makeatletter
\def\and{%
  \end{tabular}%
  \begin{tabular}[t]{c}}%
\def\@fnsymbol#1{\ensuremath{\ifcase#1\or a\or b\or c\or
   d\or e\or f\or g\or h\or i\else\@ctrerr\fi}}
\makeatother

\author{
        Marcin Pitera\,\thanks{Institute of Mathematics, Jagiellonian University, Cracow, Poland,
       \newline \hspace*{1.45em}  Email: \url{marcin.pitera@im.uj.edu.pl}, research supported by NCN grant
       2016/23/B/ST1/00479.
        \vspace{0.5em}}
\and and \ \
{\L}ukasz Stettner\,\thanks{
       Institute of Mathematics, Polish Academy of Sciences, Warsaw, Poland,
        \newline \hspace*{1.45em} Email: \url{l.stettner@impan.pl}, research supported by NCN grant
        2016/23/B/ST1/00479.
         \vspace{0.5em}}
        }

\date{ {\small%
 This version: \today}}

\usepackage{natbib}
\usepackage{placeins}

\PassOptionsToPackage{hyphens}{url}\usepackage{hyperref}
\usepackage{microtype}
\usepackage{breakcites}
\begin{document}

\maketitle

{\footnotesize
\begin{tabular}{l@{} p{350pt}}
  \hline \\[-.2em]
  \textsc{Abstract}: \ &
In this paper long-run risk sensitive optimisation problem is studied with dyadic impulse control applied to continuous-time Feller-Markov process. In contrast to the existing literature, focus is put on unbounded and non-uniformly ergodic case by adapting the weight norm approach.  In particular, it is shown how to combine geometric drift with local minorisation property in order to extend local span-contraction approach when the process as well as the linked reward/cost functions are unbounded.  For any predefined risk-aversion parameter, the existence of solution to suitable Bellman equation is shown and linked to the underlying stochastic control problem. For completeness, examples of uncontrolled processes that satisfy the geometric drift assumption are provided. \\[0.5em]
\textsc{Keywords:} \ & Impulse control, Bellman equation, non-uniformly ergodic Markov process, weight norm, risk sensitive control, entropic risk measure\\
\textsc{MSC2010:} \ & 93E20, 93C40, 60J25\\[1em]
\hline
\end{tabular}
}

\bigskip


\section{Introduction}
Let $(\Omega,\cF,\bF,\bP)$ be a continuous-time filtered probability space that satisfy the usual conditions. In
particular, we assume that $\bF=\{\cF_t\}_{t\in \bT}$, where $\bT=\bR_+$, $\cF_0$ is trivial, and
$\cF=\bigcup_{t\in\bT}\cF_t$. Moreover, let $X=(X_t)$ be a Feller-Markov process with values in a locally compact space
$E$; for simplicity we set $E=\bR^{d}$ but most results transfer directly to the general case. The process $X$
is controlled by impulses of the form $(\tau,\xi)$: at random time $\tau$ the process is shifted from the state
$X_{\tau}$ to the state $\xi$ and follows its dynamics until the next impulse. We assume that the shift $\xi$ takes
values in a compact set $U\subseteq E$. Let $\bV$ be a space of all {\it admissible} impulse control strategies
$V=\{(\tau_i,\xi_i)\}_{i=1}^{\infty}$, i.e.  sequences of strictly increasing (possibly infinite) Markov times $\tau_i$
and shift random variables $\xi_i$. Assuming that $X_0=x$ (where $x\in E$) and $V\in \bV$ we use $(\hat\Omega,\hat\cF,\bP_{(x,V)})$ to denote the probability space related to the corresponding controlled process $X$. For brevity, we omit the construction of this space; see \cite{Rob1978} for details. We refer to~\cite{PalSte2017} where a similar impulse control framework is considered and discussed in details; see also~\cite{Ste1982,Ste1989}.

The main goal of this paper is to study risk sensitive impulse control problem with reward and shift cost functions embedded in the objective function. We consider long-run version of the risk sensitive criterion with risk aversion parameter $\gamma<0$ given by
\begin{align}
J_{T}(x,V) & :=\frac{1}{\gamma}\ln \bE_{(x,V)}\left[\exp\left(\gamma{\int_0^T f(X_s)\d s +\gamma\sum_{i=1}^{\infty}
\1_{\{\tau_i\leq T\}}c(X_{\tau^-_i},\xi_i)}\right)\right],\label{eq:obj.initial}
\end{align}
defined for all $T\in\bT$, $x\in E$ and admissible controls $V\in\bV$; note that the process $X$ has
initial state $x$ and it's dynamics depends on control $V$. In \eqref{eq:obj.initial}, the function $c\colon E\times U\to\bR_{-}$ relates to the
shift execution {\it cost function}, the function $f\colon E\to \bR$ corresponds to the {\it reward function}, and $X_{\tau_i^-}$
is the state of the process before the $i$-th impulse (with a natural meaning if there is more than one impulse at the
same time).

Risk sensitive control could be seen as a non-linear extension of the risk-neutral expected cost per unit of time control studied e.g. in \cite{Rob1981,Rob1983}; see \cite{PalSte2017} for a more recent contribution in the impulse control context.  While impulse control is among the most popular forms of control, application of the standard methods in the risk sensitive case usually lead to difficult problems linked to quasi variational inequalities; see \cite{Nag2007}, and references therein. Consequently, alternative tools need to be developed; see e.g. \cite{HdiKar2011}. In this paper, we refine and extend the probabilistic approach to impulse risk sensitive control developed initially in \cite{SadSte2002} by allowing unbounded value/cost functions and non-uniform ergodicity of the underlying process. For more general background on long-run risk-sensitive control in the bounded framework see e.g. \cite{FleMcE1995} or \cite{DiMSte1999}.

We focus on the dyadic impulse control strategies where the shifts can be applied on a discrete $\delta$-dyadic time grid. By considering weighted norms, we expand the framework initiated in~\cite{HaiMat2011} and  \cite{PitSte2016}; we also refer to~\cite{SheStaObe2013,BauRie2017}, and references therein. Our approach is based on the span-contraction framework, with generic set of assumptions centred around geometric drift and local minorisation; for (alternative) vanishing discount approach see e.g. \cite{CavHer2017} and references therein. Apart from extending the span-contraction approach to the unbounded case, we also show the simple novel long-run noise control method based on application of H{\"o}lder's inequality to the underlying entropic utility. By splitting the process
into different components, and applying the entropic super and subadditive bounds (see Lemma~\ref{lm:holder}) we are able to get rid of the noise in the limit.
This simple observation allow us to quickly link the Bellman solution to the underlying optimisation problem when the noise is unbounded; see Proposition~\ref{pr:RSC.Bellman.solution2}. This method is quite general and could be used e.g. in long-run risk-sensitive portfolio optimisation. As an example, on can easily refine Proposition 5 in \cite{PitSte2016} by showing that Bellman equation always corresponds to the optimal strategy (defined therein) without any additional assumptions.

This paper is organised as follows. Section~\ref{S:preliminaries} establishes the general setup. In particular, we introduce and discuss core assumptions and state the main problem therein. In Section~\ref{S:bellman}, we introduce the dyadic Bellman equation and show that the solution to it exists. Theorem~\ref{th:1} stating that the Bellman operator is a local contraction in the shrinked $\omega$-span norm is a central part of the span-contraction approach and might be seen as one of the main results of this paper. Section~\ref{S:link} links the Bellman's equation to the corresponding dyadic optimal control problem~\eqref{eq:main1}; the main result of this section is Proposition~\ref{pr:RSC.Bellman.solution2}. In Section \ref{S:examples}, we show the reference examples of uncontrolled processes that satisfies entropic inequalities that will be introduced in Assumption \eqref{A.3}; this is important from the pragmatic point of view perspective, as the assumption might look restrictive at the first sight. Finally, in Appendix \ref{S:appendix} we introduce and prove some supplementary results including the simple proof of entropic H{\"o}lder's inequalities.

\section{Preliminaries}\label{S:preliminaries}

Let us fix $\delta>0$ ane let $\bT_{\delta}:=\{n\delta\}_{n\in\bN}$ denote the related  $\delta$-dyadic time grid. We use $\bV_{\delta}\subset \bV$ to
denote the space of all related dyadic impulse control strategies; see~\cite{SadSte2002} for details. For transparency, for a fixed
$\gamma<0$ any $n\in\bN$, we define $T_n:=n\delta$ and consider the dyadic average-cost long-run version of \eqref{eq:obj.initial} defined as
\begin{equation}\label{eq:objective.function}
J(x,V) :=\liminf_{n\to\infty}\frac{J_{T_n}(x,V)}{T_n}.
\end{equation}
While most results could be easily extended to the full time domain, considering only discrete dyadic times in \eqref{eq:objective.function} increases the transparency and is more natural when considering discrete Bellman equations; when required, we provide additional comments on how to extend our framework to full time domain.

Given the initial state $x\in E$ and impulsive control $V\in \bV$, we define the corresponding entropic utility measure
$\mu^\gamma_{(x,V)}:L^0(\hat\Omega,\hat\cF,\bP_{(x,V)})\to \bar\bR$ with risk-aversion parameter $\gamma\in \bR$
by setting
\[
\mu^\gamma_{(x,V)}(Z):=
\begin{cases}
1/\gamma\, \ln \bE_{(x,V)}\left[\exp(\gamma Z)\right] & \gamma\neq 0,\\
\bE_{(x,V)}\left[Z\right] & \gamma= 0,
\end{cases}
\]
where $\bE_{(x,V)}$ is the expectation operator corresponding to $\bP_{(x,V)}$. For brevity, we use $\mu_{x}^\gamma$ to
denote entropic utility corresponding to uncontrolled process starting at $x\in E$ (e.g. for $V\in\bV$ such that
$\tau_i=\infty$ for $i\in\bN$). If there is no ambiguity, we write $\mu^\gamma$ instead of $\mu^{\gamma}_x$ or
$\mu^\gamma_{(x,V)}$. Same applies to the probability measure $\bP_{(x,V)}$ as well as the expectation operator
$\bE_{(x,V)}$. In particular, note that \eqref{eq:obj.initial} could be rewritten as
\begin{align}
J_{T}(x,V) & =\mu^\gamma_{(x,V)}\left({\int_0^T f(X_s)\d s +\sum_{i=1}^{\infty} \1_{\{\tau_i\leq
T\}}c(X_{\tau^-_i},\xi_i)}\right).\label{eq:obj.initial2}
\end{align}

Let $\omega: E\to \bR_{+}$ be a fixed continuous {\it weight} function and let $C_{\omega}(E)$ denote the space of all
real-valued continuous functions which are bounded wrt. $\omega$-norm, i.e. functions $g\colon E\to\bR$ such that
\[
\|g\|_{\omega} :=\sup_{x\in E} \frac{|g(x)|}{1+\omega(x)}<\infty.
\]
Next, we present assumptions that will be used throughout the paper. In assumptions  \eqref{A.3}--\eqref{A.4} the
process $X=(X_t)$ corresponds to the uncontrolled process with initial
state $x\in E$.
\begin{enumerate}

\item[(\namedlabel{A.1}{A.1})] (Reward function constraints.) The function $f$ is continuous and $\|f\|_{\omega}<\infty$.
\item[(\namedlabel{A.2}{A.2})] (Shift cost function constraints.) The function $c$ is continuous and there exists
    $c_0<0$, such that for all $x\in E$ and $\xi\in U$ we get $c(x,\xi)\leq c_0$.
Moreover, $\|\hat c\|_{\omega}<\infty$, where $\hat c\colon E\to \bR_{-}$ is given by $\hat c(x):=\inf_{\xi\in
U}c(x,\xi)$.
\item[(\namedlabel{A.3}{A.3})] (Geometric drift with controllable noise.) There exist a constant $b_1\in (0,1)$,
    and (finite) functions $M_1,M_2\colon \bR\to\bR$, such that for any $\gamma\in\bR$ and $x\in E$ we get
\begin{equation}\label{eq:A3.pre}
\mu_{x}^{\gamma}\left(\int_0^\delta \omega(X_s)\,ds\right)\leq \omega(x)+M_1(\gamma) \quad\textrm{ and }\quad
\mu_{x}^{\gamma}\left(\omega(X_{\delta})\right)\leq b_1\omega(x)+M_2(\gamma).
\end{equation}
\item[(\namedlabel{A.4}{A.4})] (Local minorization.) For any $R>0$, there exists $d>0$ and probability measure
    $\nu$, such that
\begin{equation}\label{eq:rsc:unerg}
\inf_{x\in C_{R}}\bP_{x}[X_{\delta}\in A]\geq d\nu(A),\quad A\in\cB(E),
\end{equation}
where $C_{R}=\{x\in E \colon \omega(x)\leq R\}$ and $\nu$ satisfies $\nu(U)>0$.
\end{enumerate}
Let us now briefly discuss the assumptions.

Assumptions \eqref{A.1} and \eqref{A.2} are standard assumptions which allow us to operate on the space $C_{\omega}(E)$
of $\omega$-bounded functions. For technical reasons, we assume that the cost of the shift is always strictly negative
and bounded away from zero (by $c_0$); this is a classical impulse control assumption.

Assumption \eqref{A.3} relates to geometric drift property of the uncontrolled process. For simplicity, let us focus on
the second inequality. For a fixed $x\in E$ the random variable $\omega(X_{\delta})-b_1\omega(x)$ might be understood
as the {\it $\omega$-noise}, with upper bound imposed on its entropic utility. Since the distribution of
$\omega(X_{\delta})-b_1\omega(x)$ might depend on $x\in E$ we cannot split noise from the starting point as done
in~\cite{PitSte2016}; the global upper bound $M_2(\gamma)$ in \eqref{eq:A3.pre} relates to distribution level
constraints. Indeed, assuming the standard probability space and noting that entropic risk measure is law-invariant, we
can rephrase \eqref{A.3} using the concept of first-order stochastic dominance: we can assume existence of a random
variable $Z$, such that $Z$ has finite moments and stochastically dominates (positive part of)
$\omega(X_{\delta})-b_1\omega(x)$ for any $x\in E$; see \cite[Theorem 4.2]{BauMul2006} for details. In order to have
all moments finite $Z$ must belong to Orlicz heart induced by the entropic risk measure; see \cite{CheLi2009}. For
example, with $E=\bR$ and $\omega(\cdot)=|\cdot|$ assumption \eqref{A.3} holds for uncontrolled processes with dynamics
given by
\[
\d X_t=[aX_{t}+g(X_t)]\d t +\sigma(X_t)\d W_t,
\]
where $a<0$, functions $g\colon \bR\to\bR$ and $\sigma\colon \bR\to\bR_{+}$ are bounded, and $W_t$ is a standard
Brownian motion. More generally, \eqref{A.3} is satisfied for Gaussian-type of noise given e.g. via suprema of Gaussian
random vectors; we refer to Section~\ref{S:examples} for more details and to~\cite{PitSte2016} for further
discussion.

Assumption  \eqref{A.4} is a (local) minorization property. Combined with \eqref{A.3} it constitues the ergodicity
property of the underlying uncontrolled process; see \cite{HaiMat2011} for details. For bounded $\omega$ it is
equivalent to a global Doeblin's condition (uniform ergodicity), while for unbounded $\omega$ it might be linked to the
local mixing condition. Note that we additionally require that the support of invariant measure $\nu$ must have a
non-empty intersection with control (shift) set $U$.

\noindent The main goal of this paper is to find optimal control (and solution) to problem
\begin{equation}\label{eq:main1}
\sup_{V\in\bV_{\delta}}J(x_0,V),
\end{equation}
where $x_0$ is the (given) initial state.

\begin{remark}[Dyadic dynamics]\label{rem:time.grid1}
While in this paper we fix time-step $\delta>0$, it might be interesting to extend the assumptions for the general dyadic
control case. First, note that assumptions \eqref{A.1} and \eqref{A.2} are independent of the underlying choice of
$\delta$. Second, assumption \eqref{A.3} relies on the choice of $\delta$ e.g. via the {\it shrinkage constant} $b_1$
and {\it noise constraints} $M_i(\gamma)$ ($i=1,2$). Treating $b_1$ and $M_i(\gamma)$ as functions of $\delta$ and
letting $\delta\to 0$ we should get $b_1(\delta)\to 1$ and  $M_i(\gamma,\delta)\to 0$, for any $\gamma\in\bR$. Also,
assuming the noise is {\it divisible}, it would be rational to assume $\limsup_{\delta\to\infty}
M_i(\gamma,\delta)/\delta< \infty$. Finally, note that assumption \eqref{A.4} depends on the choice of the time-grid
parameter, but it would be (typically) enough to introduce dependence of $d$ and $\nu$ on $\delta$, without any
additional uniform constraints.
\end{remark}

\section{Bellman equation}\label{S:bellman}

Following~\cite{SadSte2002} and~\cite{PitSte2016} we define the Bellman equation for the dyadic impulsive control as
\begin{equation}\label{eq:rsc:bellmaneq}
w^{\gamma}_{\delta}(x)+\lambda^{\gamma}_{\delta}=\max \left\{
\mu^\gamma_{x}\left(\int_0^{\delta}f(X_s)\,ds +w^{\gamma}_{\delta}(X_{\delta})\right),
\sup_{\xi\in U}\left(\mu^{\gamma}_{\xi}\left(\int_0^{\delta}f(X_s)\,ds
+w^{\gamma}_{\delta}(X_{\delta})\right)+c(x,\xi)\right)
\right\},
\end{equation}
for $x\in E$, where $\lambda^{\gamma}_\delta\in\bR$ and $w^{\gamma}_{\delta}\in C_{\omega}(E)$. Equation
\eqref{eq:rsc:bellmaneq} can be equivalently stated as the ordinary risk sensitive discrete-time control problem
\begin{equation}\label{eq:bellmaneq2}
w^{\gamma}_{\delta}(x)+\lambda^{\gamma}_{\delta}= \sup_{a\in
\bar{U}}\left(\mu_x^{\gamma,a}\left(\int_0^{\delta}f(X_s)\,ds
+w^{\gamma}_{\delta}(X_{\delta})\right)+\bar{c}(x,a)\right),
\end{equation}
where
\[
a=(a^1,a^2)\,\,,\,\,\,a^1 \in \{0,1\} \,\,,\,\,a^2 \in
U\,\,,\,\,\bar{U}=\{0,1\}\times U\,,
\]
\[
\bar c(x,a)
 = \left\{
\begin{array}{lll}0\,\,&\mbox{if}\,\, & a^1=0,\\
c(x,\xi)\,\,&\mbox{if} \,\,&  a^1=1\,,\,a^2=\xi,\\
\end{array}
\right.
\]
\[
\mu_x^{\gamma,a}=\left\{
\begin{array}{lll}
\mu^\gamma_{\xi}&\quad \mbox{if}&\quad a=(1,\xi),\\
\mu^\gamma_x& \quad \mbox{if}& \quad a=(0,\xi).\\
\end{array}
\right.
\]
On the space $C_{\omega}(E)$\,, we define the corresponding discrete-time Bellman operator
\begin{equation}\label{Rf}
R_{\gamma}g(x):= \sup_{a\in \bar{U}}\left(\mu_x^{\gamma,a}\left(\int_0^{\delta}f(X_s)\,ds
+g(X_{\delta})\right)+\bar{c}(x,a)\right),\quad g\in \cC_{\omega}(E),
\end{equation}
and the associated operator
\[
T_{\gamma}g(x):=\gamma R_{\gamma}(g(x)/\gamma).
\]
For any $g\in C_{\omega}(E)$ and $x\in E$ we use $a_{(x,g)}$ to denote the maximiser of $T_{\gamma}g(x)$. Recalling that the Esscher transformation defines the maximising measure in the robust (dual, biconjugate) representation of the entropic utility measure (see e.g. \cite{PraMenRun1996}) for any $g\in
C_{\omega}(E)$\,, $x\in E$\,, $a\in \bar{U}$\,, and measurable set $B$, we define the associated measure
\begin{equation}\label{eq:mu.star}
\mu^*_{(x,g,a)}(B):=\frac{\bE^{a}_{x}\left[e^{\gamma\int_0^\delta f(X_s)\,ds +g(X_{\delta})}\1_{\{X_{\delta}\in B\}}
\right]}{\bE^{a}_{x}\left[e^{\gamma\int_0^\delta f(X_s)\,ds +g(X_{\delta})} \right] },
\end{equation}
where
\begin{equation}\label{eq:Exa}
\bE_x^a:=\mu_x^{0,a}=\left\{
\begin{array}{lll}
\bE_{\xi}&\quad \mbox{if}&\quad a=(1,\xi),\\
\bE_x& \quad \mbox{if}& \quad a=(0,\xi).\\
\end{array}
\right.
\end{equation}
For more details we refer to \cite{PitSte2016} where the equivalent of \eqref{eq:mu.star} is defined in Equation (29) and the dual representation of entropic utility in similar setting is discussed; see also \cite{Ger1979} for more details about Esscher transform.
\begin{proposition}\label{pr:RSC.feller}
Under assumptions \eqref{A.1}--\eqref{A.3} operators $R_{\gamma}$ and $T_{\gamma}$ transforms the set $\cC_{\omega}(E)$
into itself. Moreover, for any $g\in \cC_{\omega}(E)$ the mappings $(x,\gamma) \mapsto T_\gamma g(x)$ and  $(x,\gamma)
\mapsto R_\gamma g(x)$ are continuous on $E\times (-\infty,0)$.
\end{proposition}

\begin{proof}
We only show the proof for $R_{\gamma}$ as the proof for $T_{\gamma}$ is analogous. Let $\gamma<0$ and $g\in
\cC_{\omega}(E)$.

First, let us prove that $\| R_{\gamma}g\|_{\omega}<\infty$. For $x\in E$ we set
$F(x):=\mu^\gamma_x\left(\int_0^{\delta}f(X_s)\,ds +g(X_{\delta})\right)$. Using \eqref{A.1}, \eqref{A.3}, and
monotonicity of the entropic utility measure, for any $x\in E$ we get
\begin{align}
F(x) &\leq \mu^\gamma_x\left(\| f\|_{\omega}\int_0^{\delta}(\omega(X_s)+1)\,ds
+\|g\|_{\omega}(\omega(X_{\delta})+1)\right)\nonumber\\
& \leq \mu^\gamma_x\left(\| f\|_{\omega}\int_0^{\delta}\omega(X_s)\,ds +\| g\|_{\omega}\omega(X_{\delta})
\right)+(\delta\| f\|_{\omega}+\| g\|_{\omega}).\label{eq:itself.up.fit}
\end{align}
Now, using \eqref{A.3} and H{\"o}lder's inequality for entropic utility measure with $p=2$ (see Lemma~\ref{lm:holder}),
we know that for any $x\in E$ we get
\begin{align*}
\mu^\gamma_x\left(\| f\|_{\omega}\int_0^{\delta}\omega(X_s)\,ds +\| g\|_{\omega}\omega(X_{\delta}) \right) &\leq
\mu^{\gamma/2}_x\left(\| f\|_{\omega}\int_0^{\delta}\omega(X_s)\,ds\right) +\mu^{-\gamma}_x\left(\|
g\|_{\omega}\omega(X_{\delta}) \right)\\
& \leq (\| f\|_{\omega}+\| g\|_{\omega})\left[\omega(x)
+M_1(\gamma\|f\|_{\omega}/2)+M_2(-\gamma\|g\|_{\omega})\right].
\end{align*}
and consequently $\sup_{x\in E} \frac{F(x)}{1+\omega(x)} < \infty$. Similarly, one can show that $\inf_{x\in
E}\frac{F(x)}{1+\omega(x)} > -\infty$. Thus, we get
\begin{equation}\label{eq:FisComega}
\|F\|_\omega<\infty.
\end{equation}
Now, noting that $\omega$ is continuous and $U$ is compact, for any $x\in E$ we get
\begin{equation} \label{eq:itself1}
\sup_{\xi\in U}\left( F(\xi)+c(x,\xi)\right) \leq \|F\|_{\omega}\left(\sup_{\xi\in U}\omega(\xi)+1\right)+c_0<\infty.
\end{equation}
Combining \eqref{eq:FisComega} with \eqref{eq:itself1} we get $\sup_{x\in E}
\frac{R_{\gamma}g(x)}{1+\omega(x)}<\infty$. Then, noting that $R_{\gamma}g(x)\geq F(x)$, we get
\[
\inf_{x\in E} \frac{R_{\gamma}g(x)}{1+\omega(x)}\geq \inf_{x\in E} \frac{F(x)}{1+\omega(x)}  > -\infty,
\]
which concludes the proof of $\| R_{\gamma}g\|_{\omega}<\infty$.

Second, let us prove that the mapping $(x,\gamma) \mapsto R_{\gamma}g(x)$ is continuous on $E\times (-\infty,0)$. Fix
$\gamma<0$, $x\in E$, and let $\{(x_n,\gamma_n)\}_{n\in\bN}$ be a sequence satisfying $(x_n,\gamma_n) \to (x,\gamma)$,
$n\to\infty$, where for any $n\in\bN$ we have $(x_n,\gamma_n)\in E\times (-\infty,0)$. For $n,m\in \bN\cup\{\infty\}$ we set
\[
Z(n,m):=e^{\gamma_n \left[\int_0^{\delta}f_m(X_s)\,ds+g_m(X_{\delta})\right]},
\]
where  $f_m\colon E\to\bR$ and $g_m\colon E\to\bR$ are given by $f_m(\cdot)=(f(\cdot)\vee -m)\wedge m$ and $g_m(\cdot)=(g(\cdot)\vee -m)\wedge m$, and notation $\gamma_{\infty}:=\gamma$, $f_{\infty}(\cdot):=f(\cdot)$, and $g_{\infty}(\cdot):= g(\cdot)$ is used. Clearly, $f_m(z)\to f(z)$ and $g_m(z)\to g(z)$ for $z\in E$, as $m\to \infty$. For any $m\in\bN$, combining Feller property with the fact that
\[
\left|(\gamma_n-\gamma)\left[\int_0^{\delta}f_m(X_s)\,ds+g_m(X_{\delta})\right]\right|\leq (\delta+1)m\left|\gamma_n-\gamma\right|
\]
and $|\gamma_n-\gamma|\to 0$, as $n\to\infty$, we get
\[
\bE_{x_n}\left[ e^{\gamma_n \left[\int_0^{\delta}f_m(X_s)\,ds+g_m(X_{\delta})\right]}\right]\to \bE_{x}\left[
e^{\gamma\left[\int_0^{\delta}f_m(X_s)\,ds+g_m(X_{\delta})\right]}\right],\quad n\to\infty,
\]
which could be rewritten as
\begin{equation}\label{eq:Feller.limit1}
\bE_{x_n}\left[ Z(n,m)\right]\to \bE_{x}\left[ Z(\infty,m)\right],\quad n\to\infty.
\end{equation}
Next, we show that the class of random variables $\{Z(n,m)\}_{n,m\in\bN\cup\{\infty\}}$ is uniformly integrable on $\bP_{y}$, for any $y\in \hat V$, where $\hat V\subset E$ is a compact set such that $(\{x_n\}_{n\in\bN}\cup \{x\} \cup U) \subseteq \hat V$. Using \eqref{A.3}, for any $y\in \hat V$ and $m,n\in\bN\cup\{\infty\}$, we get
\begin{align}
\bE_{y}\left[(Z(n,m))^2\right] & \leq
\bE_{y}\left[e^{-2\gamma_n\left[\|f_m\|_{\omega}\int_0^{\delta}(\omega(X_s)+1)\,ds+\|g_m\|_{\omega}(\omega(X_{\delta})+1)\right]}\right]\nonumber\\
& \leq \bE_{y}\left[e^{-2\bar \gamma\left[\|f\|_{\omega}\int_0^{\delta}(\omega(X_s)+1)\,ds+\|g\|_{\omega}(\omega(X_{\delta})+1)\right]}\right]\nonumber\\
& = e^{-2\bar \gamma\mu_y^{-2\bar
\gamma}\left(\|f\|_{\omega}\int_0^{\delta}(\omega(X_s)+1)\,ds+\|g\|_{\omega}(\omega(X_{\delta})+1)\right)},\label{eq:UI1}
\end{align}
where $\bar\gamma:=\inf_{n\in\bN}\gamma_n$. By similar arguments as in the first part of the proof (i.e. using \eqref{A.3} and H{\"o}lder's inequalities for
entropic utility measure), recalling that $\omega$ is continuous, and $\hat V$ is compact we get
\begin{equation}\label{eq:UI2}
\sup_{y\in \hat V} \mu_y^{-2\bar
\gamma}\left(\|f\|_{\omega}\int_0^{\delta}(\omega(X_s)+1)\,ds+\|g\|_{\omega}(\omega(X_{\delta})+1)\right)<\infty.
\end{equation}
Combining \eqref{eq:UI1} with \eqref{eq:UI2}, and noting the upper bound in  \eqref{eq:UI1} is independent of $n$ and $m$, we get that the class $\{Z(n,m)\}_{n,m\in\bN\cup\{\infty\}}$ is $L^2$-bounded sequence on $\bP_y$, for any $y\in\hat V$. In particular, this implies uniform integrability of $\{Z(n,m)\}_{n,m\in\bN\cup\{\infty\}}$ on $\bP_y$ and
\begin{equation}\label{eq:Feller.limit2}
\bE_{y}\left[ Z(n,m)\right]\to \bE_{y}\left[Z(n,\infty)\right],\quad m\to\infty,
\end{equation}
for $y\in \hat V$ and $n\in \bN\cup\{\infty\}$. Moreover, since the upper $L^2$-bound in \eqref{eq:UI1} could be chosen independently of $y$ we get
\begin{equation}\label{eq:Feller.limit3}
\lim_{K\to\infty} \left(\sup_{y\in \hat V} \,\sup_{n,m\in \bN\cup\{\infty\}} \bE_{y}\left[\1_{\{|Z(n,m)|\geq K\}}|Z(n,m)|\right]\right)=0.
\end{equation}
Next, to show that
\begin{equation}\label{eq:Feller.limit4}
\bE_{x_n}\left[Z(n,\infty)\right]\to \bE_{x}\left[Z(\infty,\infty)\right],\quad n\to\infty
\end{equation}
it is enough to note that for any fixed $m\in\bN$ we get
\begin{align}
\limsup_{n\to\infty}\left|\bE_{x_n}\left[Z(n,\infty)\right]-\bE_{x}\left[Z(\infty,\infty)\right]\right|
& \leq \limsup_{n\to\infty}|\bE_{x_n}\left[Z(n,\infty)\right]-\bE_{x_n}\left[Z(n,m))\right]|\nonumber\\
&+ \limsup_{n\to\infty}|\bE_{x_n}\left[Z(n,m)\right]-\bE_{x}\left[Z(\infty,m))\right]|\nonumber\\
&+ \limsup_{n\to\infty}|\bE_{x}\left[Z(\infty,m)\right]-\bE_{x}\left[Z(\infty,\infty))\right]|.\label{eq:Feller.limit5}
\end{align}
Indeed, combining \eqref{eq:Feller.limit1}, \eqref{eq:Feller.limit2}, \eqref{eq:Feller.limit3}, with \eqref{eq:Feller.limit5}, and letting $m\to\infty$, we get \eqref{eq:Feller.limit4}, i.e. property
\[
\bE_{x_n}\left[ e^{\gamma_n \left[\int_0^{\delta}f(X_s)\,ds+g(X_{\delta})\right]}\right]\to \bE_{x}\left[
e^{\gamma\left[\int_0^{\delta}f(X_s)\,ds+g(X_{\delta})\right]}\right],\quad n\to\infty,
\]
which in turn implies $\tilde Z(x_n,\gamma_n) \to \tilde Z(x,\gamma)$, where
$\tilde Z(w,z):=\mu_{w}^{z}\left(\int_0^{\delta}f(X_s)\,ds+g(X_{\delta})\right)$. Next, noting that for any $\xi\in U$ we get
$\tilde Z(\xi,\gamma_n) \to \tilde Z(\xi,\gamma)$, and $U$ is compact, we get
\[
\max \left\{\tilde Z(x_n,\gamma_n),\sup_{\xi\in U} \tilde Z(\xi, \gamma_n) +c(x_n,\xi)\right\} \to
\max\left\{\tilde Z(x,\gamma),\sup_{\xi\in U} \tilde Z(\xi, \gamma) +c(x,\xi)\right\},
\]
from which continuity of $(x,\gamma) \mapsto R_{\gamma}g(x)$ follows.
\end{proof}

We now show that on $C_\omega(E)$ the operator  $T_\gamma$ is a local contraction under the suitable span-norm. To
ensure that property for each single step we need to shrink the original $\omega$-norm. For any $\beta>0$ the shrinked
norm $\|\cdot\|_{\beta,\omega}$ is given by
\[
\|g\|_{\beta,\omega} :=\sup_{x\in E} \frac{|g(x)|}{1+\beta\omega(x)}<\infty,\quad g\in C_{\omega}(E),
\]
while the corresponding span semi-norm is defined as
\[
\|g\|_{\beta,\omega\textrm{-span}}:=\sup_{x,y\in\bR^{k}}\frac{g(x)-g(y)}{2+\beta\omega(x)+\beta\omega(y)},\quad g\in
C_{\omega}(E).
\]
It is useful to note that for any $g\in C_{\omega}(E)$ and $\beta>0$ we get
\[
\inf_{d\in\bR}\| g+d\|_{\beta,\omega}=\|g\|_{\beta,\omega\textrm{-span}},
\]
so that the span $\omega$-norm could be considered as the {\it centered} (wrt. 0) $\omega$-norm; see~\cite[Section
3]{PitSte2016} and \cite[Section 2]{HaiMat2011} for details.

\begin{theorem}\label{th:1}
Let $\gamma<0$. Under assumptions \eqref{A.1}--\eqref{A.4}, for sufficiently small
$\beta>0$, the operator $T_{\gamma}$ is a local contraction under $\|\cdot\|_{\beta,\omega\textrm{-span}}$, i.e. there
exist functions $\beta: \bR_{+}\to (0,1)$ and $L: \bR_{+} \to (0,1)$ such that
\[
\| T_{\gamma}f_1-T_{\gamma}f_2\|_{\beta(M),\omega\textrm{-span}}\leq L(M)\|f_1-f_2\|_{\beta(M),\omega\textrm{-span}},
\]
for $f_1,f_2\in\cC_{\omega}(E)$, such that $\| f_1\|_{\omega\textrm{-span}}\leq M$ and $\| f_2
\|_{\omega\textrm{-span}}\leq M$.
\end{theorem}

\begin{proof}
For brevity, we present only the outline the proof; please see \cite[Theorem 1]{PitSte2016} for more details. The proof
will be based on three steps.\\
\\
{\it Step 1)}\, We prove that for any $g_1,g_2\in C_{\omega}(E)$ and $x,y\in E$ we get
\begin{equation}\label{eq:lemma1}
T_{\gamma}g_1(x)-T_{\gamma}g_2(x)-(T_{\gamma}g_1(y)-T_{\gamma}g_2(y)) \leq
\|g_1-g_2\|_{\beta,\omega\textrm{-span}}\|\bH^{g_1,g_2}_{x,y}\|_{\beta,\omega\textrm{-var}},
\end{equation}
where
\[
\bH^{g_1,g_2}_{x,y}:=\bar{\mu}^{*}_{(x,g_1,a_{(x,g_2)})}-\bar{\mu}^{*}_{(y,g_2,a_{(y,g_1)})},
\]
$\bar{\mu}^{*}_{(\cdot)}$ is the projection of measure $\mu^{*}_{(\cdot)}$ (given in \eqref{eq:mu.star}) on the set of
values of the processes $X$, $\|\cdot\|_{\beta,\omega\textrm{-var}}$ is the weighted total variation norm given by
\begin{equation}\label{eq:total.variation}
\|\bH \|_{\beta,\omega\textrm{-var}}
:=\int_{E}\big(1+\beta\omega(z)\big)|\bH|(dz),
\end{equation}
and $|\bH|$ is the total variation of measure $\bH$; see \cite[Section 3]{PitSte2016} for details.

First, following the proof of \cite[Lemma 1]{PitSte2016} (see also \cite[Proposition 2.2]{DiMSte1999} where similar
calculations are done for Bellman operator without impulse cost for any $g_1,g_2\in C_{\omega}(E)$ and $x,y\in E$ we get
\begin{equation}\label{eq:11}
(T_{\gamma}g_1(x)-T_{\gamma}g_2(x))-(T_{\gamma}g_1(y)-T_{\gamma}g_2(y)) \leq \int_{E}
\big[g_1(z)-g_2(z)\big]\bH^{g_1,g_2}_{x,y}(dz).
\end{equation}
Second, using \cite[Proposition 2]{PitSte2016}, we know there exists $d\in\bR$ such that
\begin{equation}\label{eq:old1}
a_{+}(d)=a_{-}(d)=\|g_1-g_2\|_{\beta,\omega\textrm{-span}},
\end{equation}
where
\[
a_{+}(d):=\sup_{z\in\bR^{k}}\frac{g_1(z)-g_2(z)+d}{1+\beta\omega(z)}\quad\textrm{and}\quad
a_{-}(d):=-\inf_{z\in\bR^{k}}\frac{g_1(z)-g_2(z)+d}{1+\beta\omega(z)}.
\]
Noting that
\[
\int_{E} \big[g_1(z)-g_2(z)\big]\bH^{g_1,g_2}_{x,y}(dz)=\int_{R}
\frac{g_1(z)-g_2(z)+d}{1+\beta\omega(z)}(1+\beta\omega(z))\bH^{g_1,g_2}_{x,y}(dz).
\]
and using the Hahn-Jordan decomposition for signed measure $\bH^{g_1,g_2}_{x,y}$, we get
\begin{equation}\label{eq:old2}
\int_{E} \big[g_1(z)-g_2(z)\big]\bH^{g_1,g_2}_{x,y}(dz) \leq
a_{+}(d)\int_{A}(1+\beta\omega(z))\bH^{g_1,g_2}_{x,y}(dz)-a_{-}(d)\int_{A^{c}}(1+\beta\omega(z))\bH^{g_1,g_2}_{x,y}(dz),
\end{equation}
where $A$ corresponds to positive set of measure $\bH^{g_1,g_2}_{x,y}$. Consequently, recalling
\eqref{eq:total.variation} and combining \eqref{eq:old1} with \eqref{eq:old2} we get \eqref{eq:lemma1}.\\
\\
{\it Step 2)}\, We prove that for any fixed $M>0$ and $\phi\in (b_1,1)$, there exists $\alpha_{\phi}>0$, such that
\begin{equation}\label{eq:beta.var.leq.var}
\|\bH^{g_1,g_2}_{x,y} \|_{\beta, \omega\textrm{-var}} \leq
\|\bH^{g_1,g_2}_{x,y} \|_{\textrm{var}}+\beta(\phi\omega(x)+\phi\omega(y)+2\alpha_{\phi}),
\end{equation}
for $x,y\in E$ and $g_1,g_2\in\cC_{\omega}(E)$ satisfying $\| f\|_{\omega\textrm{-span}}\leq M$ and  $\|
g\|_{\omega\textrm{-span}}\leq M$; $\|\cdot\|_{\textrm{var}}$ denotes the standard variation norm.

First, note that
\[
\|\bH^{g_1,g_2}_{x,y} \|_{\beta, \omega\textrm{-var}} \leq \|\bH^{g_1,g_2}_{x,y}
\|_{\textrm{var}}+\beta\left(\int_{E}\omega(z)\bar{\mu}^{*}_{(x,g_1,a_{(x,g_2)})}(dz)+\int_{E}\omega(z)\bar{\mu}^{*}_{(y,g_2,a_{(y,g_1)})}(dz)\right).
\]
Consequently, it is enough to show that there exists $\alpha_\phi>0$ such that for any $x\in E$\,, $a\in \bar{U}$\,,
and $g\in C_{\omega}(E)$ satisfying $\| g\|_{\omega\textrm{-span}}\leq M$, we get
\begin{equation}\label{eq:omegaK}
\int_{E}\omega(z) \bar{\mu}^{*}_{(x,g,a)}(dz)\leq \phi\omega(x)+\alpha_{\phi};
\end{equation}
note that for $a\in \{1\}\times U$ the term $\phi\omega(x)$ is added artificially for consistency purposes and does not
relate to state after applying the shift, i.e. since $\omega$ is bounded on the compact set $U$, for any $\xi\in U$ the
term $\phi\omega(\xi)$ could be included in $\alpha_{\phi}$ by increasing the constant by $\phi\sup_{\xi\in
U}\omega(\xi)$. Thus, setting $Z:=\gamma\int_0^\delta f(X_s)\,ds +g(X_{\delta})$, recalling \eqref{eq:mu.star}, and
noting that
it is sufficient to consider $a\in \{0\}\times U$ since $U\subset E$, we can rewrite inequality~\eqref{eq:omegaK} as
\begin{equation}\label{eq:step2.1}
\bE_x\left[\left(\omega(X_{\delta})-\phi\omega(x)\right)e^{Z}\right] \leq
\alpha_{\phi}\bE_x\left[e^{Z}\right].
\end{equation}
Let $K:=M-\delta\gamma \|f\|_{\omega}$. Multiplying both sides of~\eqref{eq:step2.1} by $\frac{2K}{\phi-b_1}$, noting
that $y<e^y$ for $y\in\bR$, and taking logarithm on both sides it is enough to show
\[
\ln\bE_x\left[e^{\frac{2K}{\phi-b_1}\left(\omega(X_{\delta})-\phi\omega(x)\right)}e^{Z}\right] \leq \ln
\frac{K\alpha_{\phi}}{\phi-b_1} +\ln\bE_x\left[e^{Z}\right],
\]
which is equivalent to
\begin{equation}\label{eq:big0}
\mu^{1}_{x}\left(\frac{2K}{\phi-b_1}(\omega(X_{\delta})-b_1\omega(x))+Z+d\right)-\mu^{1}_{x}(Z+d) \leq \ln
\frac{K\alpha_{\phi}}{\phi-b_1}+2K\omega(x),
\end{equation}
where $d\in\bR$ is  (centralizing constant) such that $\| g+d \|_{\omega}\leq M$. Noting that
\begin{align*}
Z +d&=\gamma\int_0^\delta f(X_s)\,ds +g(X_{\delta})+d\\
& \geq \gamma \|f\|_{\omega}\int_0^\delta \left[\omega(X_s)+1\right]\,ds-M[\omega(X_\delta)+1]\\
& \geq -K+\gamma \|f\|_{\omega}\int_0^\delta \omega(X_s)\,ds-M\omega(X_\delta),
\end{align*}
using H{\"o}lder's inequality for entropic utility measure with $p=2$ (see Lemma~\ref{lm:holder}), and recalling
$\eqref{A.3}$ we get
\begin{equation}\label{eq:big1}
-\mu^1_{x}(Z+d)\leq K(\omega(x)+1)-\gamma \|f\|_{\omega} M_1\left(\frac{\gamma \|f\|_{\omega}}{2}\right)+M \cdot
M_{2}(M).
\end{equation}
Similarly,
\begin{equation}\label{eq:big2}
\mu^{1}_{x}\left(\frac{2K}{\phi-b_1}(\omega(X_{\delta})-b_1\omega(x))+Z+d\right) \leq
\mu^{2}_{x}\left(\frac{2K}{\phi-b_1}(\omega(X_{\delta})-b_1\omega(x))\right)+\mu^{2}_{x}(Z+d),
\end{equation}
where
\begin{align}
\mu^{2}_{x}\left(\frac{2K}{\phi-b_1}(\omega(X_{\delta})-b_1\omega(x))\right) & \leq \frac{2K}{\phi-b_1}\cdot
M_{2}\left(\frac{4K}{\phi-b_1}\right);\nonumber\\
\mu^{2}_{x}(Z+d) &\leq K(\omega(x)+1) -\gamma \|f\|_{\omega}\cdot M_{1}(-4\gamma \|f\|_{\omega})+M\cdot
M_{2}(4M).\label{eq:big3}
\end{align}
Combining \eqref{eq:big1}, \eqref{eq:big2}, and \eqref{eq:big3} with \eqref{eq:big0} we know it is enough to choose
(large) $\alpha_{\phi}$ satisfying
\begin{align*}
\alpha_{\phi} &\geq \exp\Bigg( 2+\frac{2}{\phi-b_1} M_{2}\left(\frac{4K}{\phi-b_1}\right)-\frac{\gamma
\|f\|_{\omega}}{K}\left(M_{1}\left(\frac{\gamma \|f\|_{\omega}}{2}\right)+M_{1}(-4\gamma \|f\|_{\omega})\right)\\
&\phantom{\geq}\qquad\qquad+\frac{M}{K}(M_{2}(M)+M_{2}(4M))+\frac{\ln(\phi-b_1)}{K}\Bigg).
\end{align*}
This concludes the proof of \eqref{eq:beta.var.leq.var}.\\
\\
{\it Step 3)}\, Finally, we want to show that for any fixed $M>0$, $\phi\in (b_1,1)$ and $\alpha_{\phi}>0$, there
exists $\beta\in (0,1)$ and $L\in (0,1)$ such that
\begin{equation}\label{eq:rsc:var.leq}
\|\bH^{g_1,g_2}_{x,y} \|_{\textrm{var}}+\beta(\phi\omega(x)+\phi\omega(y)+2\alpha_{\phi})\leq
L(2+\beta\omega(x)+\beta\omega(y)),
\end{equation}
for any $x,y\in E$ and $g_1,g_2\in\cC_{\omega}(E)$ satisfying $\| f\|_{\omega\textrm{-span}}\leq M$ and $\| g
\|_{\omega\textrm{-span}}\leq M$.

Let us fix $M>0$, $\phi\in (b_1,1)$ and $\alpha_{\phi}>0$, and consider $R\in\bR$ such that
$R>\frac{2\alpha_{\phi}}{1-\phi}$. If $x,y\in E$ are such that $\omega(x)+\omega(y)>R$ then one could show that for any
$\beta<1$ and
\begin{equation}\label{eq:L1}
L\in \left(\max\left\{\phi,\frac{2+\beta(2\alpha_{\phi}+\phi R)}{2+\beta R}\right\},1\right)
\end{equation}
the inequality~\eqref{eq:rsc:var.leq} will hold; see proof of \cite[Lemma 3]{PitSte2016} for details. On the other
hand, if $x,y\in E$ are such that  $\omega(x)+\omega(y)\leq R$ then we can exploit the classical span-contraction
methodology for the bounded case; see e.g.~\cite{Ste1999}. Indeed, following the proof of \cite[Lemma 3]{PitSte2016} it
is enough to show that
\begin{equation}\label{eq:rsc:var.leq2}
\sup_{(x,y)\in \bar C_{R}}\|\bH^{g_1,g_2}_{x,y} \|_{\textrm{var}}<2,
\end{equation}
where $\bar C_{R}:=\{(x,y)\in E\times E: \omega(x)+\omega(y)\leq R\}$, and consider any
\begin{equation}\label{eq:L2}
L\in\left( \frac{\sup_{(x,y)\in C_{R}}\|\bH^{g_1,g_2}_{x,y} \|_{\textrm{var}}+\beta (\phi
R+2\alpha_{\phi})}{2},1\right),
\end{equation}
for some fixed $\beta\in (0,1)$ satisfying
\begin{equation}\label{eq:beta.const}
\beta<\frac{2-\sup_{(x,y)\in C_{R}}\|\bH^{g_1,g_2}_{x,y} \|_{\textrm{var}}}{\phi R+2\alpha_\phi}.
\end{equation}
The proof of~\eqref{eq:rsc:var.leq2} is based on contradiction. Assume there exists a sequence
\[(x_n,y_n,f_n,g_n,A_n)_{n\in\bN},\]
where $(x_n,y_n)\in \bar C_{R}$\,, $f_n,g_n\in\cC_{\omega}(E)$, and $A_n\in\cB(E)$ are such that $\|
f_n\|_{\omega\textrm{-span}}\leq M$, $\|
g_n
\|_{\omega\textrm{-span}}\leq M$, and $\bH^{f_n,g_n}_{x_n,y_n}(A_n)\rightarrow 1$ (as $n\to\infty$). Following the
proof of \cite[Lemma 3]{PitSte2016} for any $x\in E$, $a\in \bar U$, $g\in\cC_{\omega}(E)$ and $A\in\cB(E)$, such that
$\omega(x)\leq R$ and $\|
f\|_{\omega\textrm{-span}}\leq M$, we get
\begin{align}
\bar{\mu}^{*}_{(x,g,a)}(A) & \geq \frac{\bE_x^a\big[\1_{\set{X_\delta\in A}}\big]^{2}}{\bE_x^a[e^{\gamma\int_0^\delta
f(X_s)\,ds +g(X_{\delta})}]\bE_x^a[(e^{\gamma\int_0^\delta f(X_s)\,ds +g(X_{\delta})})^{-1}]}\nonumber\\
& \geq \frac{\bE_x^a\big[\1_{\set{X_\delta\in A}}\big]^{2}}{e^{2(M-\gamma
\delta\|f\|_{\omega})}\bE_x^a[e^{Z_2}]^{2}}\nonumber\\
& \geq \frac{\bE_x^a\big[\1_{\set{X_\delta\in A}}\big]^{2}}{e^{2(M-\gamma
\delta\|f\|_{\omega})}\bE_x^a[e^{Z_2}]^{2}}\label{eq:rsc2:schwarz},
\end{align}
where $Z_2:=-\gamma \|f\|_{\omega}\int_0^\delta \omega(X_s)\,ds+M \omega(X_\delta)$. Using similar reasoning as in
\eqref{eq:big1} and recalling \eqref{A.3} we get
\begin{equation}\label{eq:Z2b}
\bE_x^a[e^{Z_2}]^{2} \leq \exp\left(2K\max\{\omega(x)\,,\, \sup_{\xi\in U}\omega(\xi)\}+D\right),\qquad x\in E,
\end{equation}
where $K=M-\delta\gamma \|f\|_{\omega}$ and $D\in\bR$ is some fixed constant. Consequently, we get
\[
\sup_{x\in \bar C_{R}}\bE_x^a[e^{Z_2}]^{2} \leq \exp\left(2K\max\{\omega(R)\,,\, \sup_{\xi\in
U}\omega(\xi)\}+D\right).
\]
Thus, combining~\eqref{eq:rsc2:schwarz} with the fact that  $\bH^{f_n,g_n}_{x_n,y_n}(A_n)\rightarrow 1$ we get
\[
\bE_{x_n}^{a_{(x_n,g_n)}}\big[\1_{\set{X_\delta\in A^{c}_n}}\big]\to 0\quad\textrm{and}\quad
\bE_{y_n}^{a_{(y_n,f_n)}}\big[\1_{\set{X_\delta\in A_n}}\big]\to 0.
\]
On the other hand, from \eqref{A.4}, for any $n\in\bN$ and $(x_{n},y_{n})\in \bar C_{R}$, we get
\[
\bE_{x_n}^{a_{(x_n,g_n)}}\big[\1_{\set{X_\delta\in A^{c}_n}}\big]+\bE_{y_n}^{a_{(y_n,f_n)}}\big[\1_{\set{X_\delta\in
A_n}}\big]\geq c\nu(A_{n}^{c})+c\nu(A_{n})=c>0,
\]
which leads to contradiction.

Combining steps 1), 2), and 3), we conclude the proof.
\end{proof}
\noindent Next, we show that the iterated sequence $(T^n_{\gamma}0)_{n=1}^{\infty}$ is bounded in $\omega$-span
semi-norm.

\begin{proposition}\label{pr:bounded}
For any $\gamma <0$ there exists $M\in\bR_{+}$ such that
\[
\| T^{n}_{\gamma}0\|_{\omega\textrm{-span}}\leq M,\quad \textrm{for } n\in\bN.
\]
\end{proposition}

\begin{proof}
Let $\gamma <0$. For brevity, we use the notation $g_n:=R^n_\gamma 0$ with the convention $g_0\equiv0$. Moreover, we
define $x_n^*:=\argmax_{x\in U} g_{n}(x)$ and  $Z:=\int_0^{\delta}f(X_s)\,\d s$. Then, for any $n\in\bN$  and $\beta>0$
we get
\begin{align}
\|g_{n+1}\|_{\beta,\omega\textrm{-span}}
 & = \sup_{x,y\in E}\frac{ \sup_{a\in \bar{U}}\left[\mu_x^{\gamma,a}\left(Z+g_n(X_{\delta})\right)+\bar{c}(x,a)\right]-
 \sup_{a\in
 \bar{U}}\left[\mu_y^{\gamma,a}\left(Z+g_n(X_{\delta})\right)+\bar{c}(y,a)\right]}{2+\beta\omega(x)+\beta\omega(y)}\nonumber\\
  & \leq \max\left\{ K^1_\beta\,,\,\sup_{x,y\in
  E}\frac{\mu_x^{\gamma}\left(Z+g_n(X_{\delta})\right)-\mu_{x^*_n}^{\gamma}\left(Z+g_n(X_{\delta})\right)-c(y,x_n^*)}{2+\beta\omega(x)+\beta\omega(y)}\,\right\},\label{eq:big1}
 \end{align}
 where
  \[
 K^1_\beta:= \sup_{x,y\in E}\sup_{\xi\in U} \frac{c(x,\xi)-c(y,\xi)}{2+\beta\omega(y)};
 \]
note that in \eqref{eq:big1} we used the following shift strategy: if a shift is applied to the process starting in $x$ then the same shift is applied to the process starting in $y$ with $K^1_\beta$ corresponding the the upper value bound; if no shift is applied to the process starting in $x$ then the shift to $x_n^*$ is applied to the process starting in $y$. Using \eqref{A.2}, for any $\xi\in U$ and $y\in E$ we get $c(x,\xi) <c_0$ and $c(y,\xi) \geq -\|\hat
c\|_{\beta,\omega}(1+\beta\omega(y))$. Consequently, for any $\beta>0$ we have $ K^1_\beta<\infty$ and we can rewrite \eqref{eq:big1} as
\begin{equation}\label{eq:gn1}
\|g_{n+1}\|_{\beta,\omega\textrm{-span}}\leq \max\left\{ K^1_{\beta}\,,\, \sup_{x\in
E}\frac{\mu_x^{\gamma}\left(Z+g_n(X_{\delta})\right)-\mu_{x^*_n}^{\gamma}\left(Z+g_n(X_{\delta})\right)}{2+\beta\omega(x)}+\|\hat
c\|_{\beta,\omega} \right\}.
\end{equation}
Noting that for any $x\in E$  we have
\begin{equation}\label{eq:444}
g_{n}(x)  \geq  \mu_{x^*_n}^{\gamma}\left(Z +g_{n-1}(X_{\delta})\right)+c(x,x^*_n) \geq g_{n}(x_n^*) -\|\hat
c\|_{\beta,\omega}(1+\beta\omega(x)),
\end{equation}
we get
\[
\mu_{x^*_n}^{\gamma}\left(Z+g_n(X_{\delta})\right) \geq g_n(x_n^*)+\mu_{x^*_n}^{\gamma}\left(Z-\|\hat
c\|_{\beta,\omega}(1+\beta\omega(X_\delta))\right).
\]
Applying H{\"o}lder's inequality for entropic utility measure with $p=2$ (see Lemma~\ref{lm:holder}) we know that
\[
\mu_{x^*_n}^{\gamma}\left(Z-\|\hat c\|_{\beta,\omega}(1+\beta\omega(X_\delta))\right) \geq
\mu_{x^*_n}^{2\gamma}\left(Z\right)+\mu_{x^*_n}^{2\gamma}\left(-\|\hat
c\|_{\beta,\omega}(1+\beta\omega(X_\delta))\right)
\]
where, due to \eqref{A.3},
\begin{align*}
\mu_{x^*_n}^{2\gamma}\left(Z\right) &\geq
\mu_{x^*_n}^{2\gamma}\left(-\|f\|_{\beta,\omega}\int_0^{\delta}(1+\beta\omega(X_s))\d s\right)\\
 & \geq
 -\beta\|f\|_{\beta,\omega}\left[\omega(x_n^*)+M_1(-2\gamma\beta\|f\|_{\beta,\omega})\right]-\delta\|f\|_{\beta,\omega}\\
  & \geq -\|f\|_{\beta,\omega}[\sup_{\xi\in U}\omega(\xi)+M_1(-2\gamma\beta\|f\|_{\beta,\omega})+\delta]\,,\\
 \mu_{x^*_n}^{2\gamma}\left(-\|\hat c\|_{\beta,\omega}(1+\beta\omega(X_\delta))\right) & \geq  -\beta\|\hat
 c\|_{\beta,\omega}\left[\omega(x_n^*)+M_2(-2\gamma\beta\|\hat c\|_{\beta,\omega}) \right] -\|\hat
 c\|_{\beta,\omega}\\
 & \geq  -\|\hat c\|_{\beta,\omega}[\sup_{\xi\in U}\omega(\xi)+M_2(-2\gamma\beta\|\hat c\|_{\beta,\omega})+1]\,.
\end{align*}
Thus, setting
\[
K^2_\beta:=-(\|f\|_{\beta,\omega}+\|\hat c\|_{\beta,\omega})\left[\sup_{\xi\in
U}\omega(\xi)+M_1(-2\gamma\beta\|f\|_{\beta,\omega})+M_2(-2\gamma\beta\|\hat c\|_{\beta,\omega})+1+\delta\right]
\]
and introducing $c_n :=\inf_{c\in\bR} \|g_n+c\|_{\beta,\omega}$ we can rewrite \eqref{eq:gn1} as
\begin{equation}\label{eq:gn3}
\|g_{n+1}\|_{\beta,\omega\textrm{-span}} \leq \max\left\{ K^1_{\beta}\,,\, \sup_{x\in E}W_n(x)+K^2_\beta\right\},
\end{equation}
where
\[
W_n(x):=\frac{\mu_x^{\gamma}\left(Z+g_n(X_{\delta})+c_n\right)}{2+\beta\omega(x)}-\frac{g_n(x_n^*)+c_n}{2+\beta\omega(x)}.
\]
Next, using the fact that entropic risk measure is increasing with respect to the risk-averse parameter $\gamma$,
noting that $\|g_n+c_n\|_{\beta,\omega}=\| g_n\|_{\beta,\omega\textrm{-span}}$, and using assumptions
\eqref{A.1}--\eqref{A.3}, we get
\begin{align*}
\frac{\mu_x^{\gamma}\left(Z+g_n(X_{\delta})+c_n\right)}{2+\beta\omega(x)} &\leq \frac{\mu_x^{0}\left(Z+\|
g_n\|_{\beta,\omega\textrm{-span}}(1+\beta\omega(X_\delta))\right)}{2+\beta\omega(x)}\\
& \leq  \frac{\bE_x\left[Z+\|
g_n\|_{\beta,\omega\textrm{-span}}(1+\beta(b_1\omega(x)+M_2(0)))\right]}{2+\beta\omega(x)}\\
& \leq \frac{1+\beta b_1\omega(x)+\beta M_2(0)}{2+\beta\omega(x)}\| g_n\|_{\beta,\omega\textrm{-span}}
+\frac{\bE_x\left[Z\right]}{1+\beta\omega(x)}\\
& \leq \frac{1+\beta b_1\omega(x)+\beta M_2(0)}{2+\beta\omega(x)}\|
g_n\|_{\beta,\omega\textrm{-span}}+\|f\|_{\beta,\omega}\bE_x\left[\frac{\int_0^{\delta}(1+\beta\omega(X_\delta))\d
s}{1+\beta\omega(x)}\right]\\
& \leq \frac{1+\beta b_1\omega(x)+\beta M_2(0)}{2+\beta\omega(x)}\| g_n\|_{\beta,\omega\textrm{-span}}
+\|f\|_{\beta,\omega} (\delta+\beta M_1(0)).
\end{align*}
Now, let us fix $\beta:=(2\sup_{\xi\in U} \omega(\xi))^{-1}$. Then, we get
\begin{align*}
W_n(x)& \leq \frac{1+\beta b_1\omega(x)+\beta M_2(0)}{2+\beta\omega(x)}\| g_n\|_{\beta,\omega\textrm{-span}}
+\|f\|_{\beta,\omega}(\delta+\beta M_1(0))-\frac{g_n(x_n^*)+c_n}{2+\beta\omega(x)}\\
& \leq \frac{1+\beta b_1\omega(x)+\beta M_2(0)}{2+\beta\omega(x)}\| g_n\|_{\beta,\omega\textrm{-span}}
+\|f\|_{\beta,\omega}(\delta+\beta M_1(0))+\frac{\|
g_n\|_{\beta,\omega\textrm{-span}}(1+\beta\omega(x_n^*))}{2+\beta\omega(x)}\\
& \leq \frac{1+\beta b_1\omega(x)+\beta M_2(0)}{2+\beta\omega(x)}\| g_n\|_{\beta,\omega\textrm{-span}}
+\|f\|_{\beta,\omega}(\delta+\beta M_1(0))+\frac{3}{4+2\beta\omega(x)}\| g_n\|_{\beta,\omega\textrm{-span}}\\
&\leq \frac{5+2\beta b_1\omega(x)+2\beta M_2(0)}{4+2\beta\omega(x)}\|
g_n\|_{\beta,\omega\textrm{-span}}+\|f\|_{\beta,\omega}(\delta+\beta M_1(0)).
\end{align*}
Consequently, there exists $R>0$ such that for any $n\in\bN$ and $x\in E$ satisfying $\omega(x) >R$ we get
\begin{equation}\label{eq:geom1}
W_n(x)\leq \left(b_1+\frac{1-b_1}{2}\right)\| g_n\|_{\beta,\omega\textrm{-span}} +\|f\|_{\beta,\omega}(\delta+\beta
M_1(0)).
\end{equation}
Next, we show that there exist a constant $K^\beta_3>0$ such that for any $n\in\bN$ and $x\in C_R$, where $C_R=\{x\in
E\colon \omega(x)\leq R\}$, we get
\begin{equation}\label{eq:K3}
W_n(x)\leq K^\beta_3.
\end{equation}
Using assumption \eqref{A.4} we know that there exists $\epsilon>0$ such that for any $n\in\bN$ and $x\in C_R$ we get
$\bP_x\left[X_{\delta}\in U\right] >\epsilon$. Moreover, noting that for any $y\in U$ we have $g_n(y) \leq g_n(x_n^*)$
and that the entropic utility measure is concave for $\gamma<0$ (which implies $a\mu_{x}^{\gamma}(\cdot)\leq \mu^{\gamma}_{x}(a \cdot)$ for $a\in (0,1)$), for any $n\in\bN$ and $x\in C_R$ we get
\begin{align}
W_n(x) & \leq \frac{\mu_x^{\gamma}\left(Z+g_n(X_{\delta})-g_n(x_n^*)\right)}{1+\beta\omega(x)}\nonumber\\
 & \leq \mu_x^{\gamma}\left(\1_{\{X_\delta \in U\}}\frac{Z}{1+\beta\omega(x)} + \1_{\{X_\delta \not\in
 U\}}(+\infty)\right)\nonumber\\
    & \leq \mu_x^{\gamma}\left(\1_{\{X_\delta \in U\}}\left( \|f\|_{\beta,\omega}(\delta
    +\beta\int_{0}^{\delta}(\omega(X_s)-\omega(x))\d s) \right) + \1_{\{X_\delta \not\in
    U\}}(+\infty)\right).\label{eq:432}
\end{align}
Let $Z_x:=\int_{0}^{\delta}(\omega(X_s)-\omega(x))\d s$. Due to assumption \eqref{A.3}, we know that
\[
\sup_{x\in E} \bE_x\left[ Z_x\right]\leq M_1(0) <\infty.
\]
Thus, we know that there exists $N\in\bR$ such that
\begin{equation}\label{eq:N.constant}
\inf_{x\in C_R} \bP_{x}\left[\left\{X_\delta\in U\right\}\cap\left\{Z_x\leq N\right\}\right] \geq\epsilon/2.
\end{equation}
Combining \eqref{eq:432} with \eqref{eq:N.constant}, for any $x\in C_R$ we get
\begin{align*}
W_n(x) & \leq \mu_x^{\gamma}\left(\1_{\{X_\delta\in U\}\cap\{Z_x\leq N\}}\left( \|f\|_{\beta,\omega}(\delta +\beta
N)\right) + \1_{\{X_\delta\not\in U\}\cup\{Z_x> N\}}(+\infty)\right)\label{eq:432}\\
&  \leq \tfrac{1}{\gamma}\ln\tfrac{\epsilon}{2}+ \|f\|_{\beta,\omega}(\delta +\beta N).
\end{align*}
Consequently, setting $K_3^\beta:= \tfrac{1}{\gamma}\ln\tfrac{\epsilon}{2}+ \|f\|_{\beta,\omega}(\delta +\beta N)$ we
conclude the proof of \eqref{eq:K3}.

Next, combining \eqref{eq:geom1} and \eqref{eq:K3} we know that for any $n\in\bN$ and $x\in E$ we get
\begin{equation}
W_n(x)\leq a\| g_n\|_{\beta,\omega\textrm{-span}} +K^4_\beta,
\end{equation}
for constant parameters $a<1$ and $K^4_\beta\in \bR_{+}$. Consequently, we can rewrite \eqref{eq:gn3} as
\begin{equation}\label{eq:gn4}
\|g_{n+1}\|_{\beta,\omega\textrm{-span}} \leq \max\left\{ K^1_{\beta}\,,\, a\|
g_n\|_{\beta,\omega\textrm{-span}}+K^4_\beta+K^2_\beta\right\}.
\end{equation}
Using the standard geometric convergence arguments we know that \eqref{eq:gn4} implies existence of a constant
$M_{\beta}\in\bR_{+}$ such that for any $n\in\bN$ we get
\[
\|g_{n+1}\|_{\beta,\omega\textrm{-span}} \leq M_{\beta}.
\]
Finally, the equivalence of semi-norms $\|\cdot\|_{\beta,\omega\textrm{-span}}$ and $\|\cdot\|_{\omega\textrm{-span}}$
combined with the property
\[
\|R_\gamma^n0\|_{\omega\textrm{-span}}=|\gamma|\cdot \|T_\gamma^n0\|_{\omega\textrm{-span}}
\]
concludes the proof.
\end{proof}

Combining Theorem~\ref{th:1} with Proposition~\ref{pr:bounded}, and using Banach's fixed point theorem, we get the
solution to Bellman equation~\eqref{eq:rsc:bellmaneq}; see Proposition~\ref{pr:bellman.solution}. For brevity, we omit
the proof; see second part of the proof in~\cite[Proposition 4]{PitSte2016} for details. Note that due to
Proposition~\ref{pr:bounded} we get solution to Bellman equation for any predefined $\gamma<0$. In particular, in
contrast to \cite[Proposition 4]{PitSte2016}, we do not require $\gamma$ to be close to $0$.

\begin{proposition}\label{pr:bellman.solution}
Let $\gamma<0$. Under assumptions \eqref{A.1}--\eqref{A.4} there exist a unique (up to an additive constant)
$w^{\gamma}_{\delta}\in\cC_{\omega}(E)$ and $\lambda^{\gamma}_{\delta}\in\bR$, the solutions to Bellman
equation~\eqref{eq:rsc:bellmaneq}.
\end{proposition}

\section{Solution to the dyadic optimal control problem}\label{S:link}

Before we link the Bellman's equation to the corresponding dyadic optimal control problem~\eqref{eq:main1}, let us show
some supplementary results

\begin{proposition}\label{pr:lambda.continuity}
The mapping $\gamma \to \lambda^{\gamma}_{\delta}$ is continuous on $(-\infty,0)$.
\end{proposition}

\begin{proof}
Let us fix $a\in E$, and for any $\gamma<0$ set
\[
\bar{w}^\gamma_\delta(x):=w^\gamma_\delta(x)-w^\gamma_\delta(a),\quad x\in E.
\]
Note that $\bar{w}^\gamma_\delta$ is also a solution to Bellman equation~\eqref{eq:rsc:bellmaneq}, and from
Proposition~\ref{pr:bounded} we get $\|\gamma \bar{w}^\gamma_\delta\|_{\omega\textrm{-span}}\leq M$, where $M\in\bR_+$
is a fixed constant. Moreover, since constant $M$ in Proposition~\ref{pr:bounded} can be chosen uniformly on any
compact subset of negative $\gamma$s, say $G$, for any $x\in E$, $m\in\bN$, and $\gamma\in G$, using
Theorem~\ref{th:1}, we get
\begin{equation}\label{eq:estim}
|T_{\gamma}^m0(x)-T_{\gamma}^m0(a) -\gamma \bar{w}^\gamma_\delta(x)|\leq M (L(M))^m(2+\omega(x)+\omega(a)).
\end{equation}
Let us fix $x\in E$. By Proposition \ref{pr:RSC.feller}, the mappings $\gamma \to T_{\gamma}^m0(x)$ and $\gamma \to
T_{\gamma}^m0(a)$ are continuous for any $m\in\bN$.
Therefore, using \eqref{eq:estim}, for any $\gamma<0$, $m\in\bN$, and a sequence $(\gamma_n)_{n\in\bN}$, such that
$\gamma_n \to \gamma$, as $n\to\infty$, we get
\begin{align*}
|\gamma_n \bar{w}^{\gamma_n}_\delta(x)-\gamma \bar{w}^\gamma_\delta(x) | & \leq
|T_{\gamma_n}^m0(x)-T_{\gamma}^m0(x)|+|T_{\gamma_n}^m0(a)-T_{\gamma}^m0(a)|\\
&\phantom{\leq}+2 M (L(M))^m(2+\omega(x)+\omega(a))\\
&=a_{n,m}+b_{n,m}+c_m.
\end{align*}
For any $\epsilon>0$ we can choose $m_\epsilon\in\bN$, such that $c_{m_\epsilon}\leq \epsilon$. Consequently, letting $n\to
\infty$ with a fixed $m_{\epsilon}$, we get $\limsup_{n\to\infty}|\gamma_n \bar{w}^{\gamma_n}_\delta(x)-\gamma
\bar{w}^\gamma_\delta(x) |\leq\epsilon$. As the choice of $\epsilon$ is arbitrary, we get continuity of the mapping
$\gamma \to \gamma \bar{w}^\gamma_\delta(x)$. Next, following the proof of Proposition~\ref{pr:RSC.feller}, we see that
the mapping $\gamma \to  T_{\gamma}\gamma\bar{w}^\gamma_\delta(x)$ is also continuous. Consequently, noting that
\[
\gamma\lambda^{\gamma}_{\delta}=T_\gamma \gamma\bar{w}^\gamma_\delta(x)-\gamma\bar{w}^\gamma_\delta(x),
\]
and using similar arguments as in~\cite[Proposition 4.8]{PitSte2016}, we obtain continuity of $\gamma\to \gamma
\lambda^{\gamma}_{\delta}$ on $(-\infty,0)$. This implies continuity of $\gamma\to \lambda^{\gamma}_{\delta}$ on
$(-\infty,0)$, and completes the proof.
\end{proof}

\begin{proposition}\label{pr:w.omega}
For any $\gamma\in\bR$ and $x\in E$, we get
\begin{equation}\label{eq:w.omega}
\sup_{V\in\bV_{\delta}}\sup_{t\in\bT_{\delta}} \mu^\gamma_{(x,V)}(\omega(X_{t})) <\infty.
\end{equation}
\end{proposition}

\begin{proof}
 Let us fix $\gamma\in\bR$. Let $b_2\colon E\times E\to \bR_{+}$ be given by
 \begin{equation}\label{eq:b2}
 b_2(z,y):=[\omega(z+y)-b_1\omega(z)]_{+}.
 \end{equation}
In particular, note that for any $x\in E$ we get $\omega(X_{\delta})\leq b_1\omega(x)+b_2(x,X_\delta-x)$ and
\begin{equation}\label{eq:tildeM1}
\tilde M_2(|\gamma|):=\sup_{x\in E}\mu_x^{|\gamma|}\left(b_2(x,X_{\delta}-x)\right) <\infty.
\end{equation}
For completeness, let us outline the proof of \eqref{eq:tildeM1}. On the first hand, note that for any sequence
$(x_n)_{n\in\bN}$, where $x_n\in E$, taking the limit $n\to\infty$, we get
$\mu_{x_n}^{|\gamma|}\left(\omega(X_{\delta})-b_1\omega(x_n)\right) \to\infty$ if and only if
$\bE_{x_n}\left[e^{|\gamma|(\omega(X_{\delta})-b_1\omega(x_n))}\right]\to\infty$. Consequently, since function $z\mapsto e^z$
is bounded from below, we get $\mu_{x_n}^{|\gamma|}\left(\omega(X_{\delta})-b_1\omega(x_n)\right) \to\infty$ if and
only if $\mu_{x_n}^{|\gamma|}\left(b_2(x_n,X_{\delta}-x_n)\right) \to\infty$. On the other hand, using \eqref{A.3}, we
get $\mu_{x_n}^{|\gamma|}\left(\omega(X_{\delta})-b_1\omega(x_n)\right) \leq M_{2}(|\gamma|)$. These two facts imply
\eqref{eq:tildeM1}.

 Now, we fix $x\in E$ and introduce some additional auxiliary notation. Let $a:=\sup_{x\in U}\omega(x)$ and for any
 $n\in\bN$ let
 \begin{align*}
 A_n &:=\{X_{n\delta}\not\in U\},\\
  \cF_{n} & :=\sigma(X_s, s\in [0,i\delta]),\\
 B^n_i &:=b_2(X_{(n-i-1)\delta},X^{-}_{(n-i)\delta}-X_{(n-i-1)\delta}),\quad i=0,1,\ldots,n-1,
 \end{align*}
where $X_{t}^-$ is the state of $(X_t)$ before the (optional) shift; note that
$\1_{A_n}X_{n\delta}=\1_{A_n}X^{-}_{n\delta}$ for $n\in \bN$. For brevity, we also use $\mu^{\gamma}_{(x,V)}(\cdot \mid
\cF_i)$ to denote the $\cF_i$-conditional equivalent of $\mu^{\gamma}_{(x,V)}$.

Let us fix $V\in\bV_{\delta}$ and $t\in\bT_{\delta}$. Noting that $t=n\delta$ for some $n\in\bN$, using monotonicity of
$\mu^{\gamma}_{(x,v)}$, and \eqref{A.3}, we get
\begin{align}
 \mu^\gamma_{(x,V)}(\omega(X_{n\delta})) &\leq
 \mu^\gamma_{(x,V)}(\1_{A^{'}_n}a+\1_{A_n}\omega(X_{n\delta}))\nonumber\\
& \leq
\mu^\gamma_{(x,V)}\left(\1_{A^{'}_n}a+\1_{A_n}\left[b_1\omega(X_{(n-1)\delta})+b_2(X_{(n-1)\delta},X^{-}_{n\delta}-X_{(n-1)\delta})\right]\right)\nonumber\\
&\leq \ldots\nonumber\\
& \leq \mu^\gamma_{(x,V)}\left(\omega(x)+a+\sum_{i=0}^{n-1}\1_{\bigcap_{j=0}^{i}
A_{n-j}}b^{i}_1 B^n_i\right)\nonumber\\
& \leq \omega(x)+a+\mu^\gamma_{(x,V)}\left(\sum_{i=0}^{n-1}b^{i}_1B^n_i\right).\label{eq:pr1}
\end{align}
Using strong time-consistency and additivity of entropic utility, we have
\begin{equation}\label{eq:pr2}
\mu^\gamma_{(x,V)}\left(\sum_{i=0}^{n-1}b^{i}_1B^n_i\right)
\leq\mu^\gamma_{(x,V)}\left(\mu^\gamma_{(x,V)}\left(\sum_{i=0}^{n-1}b^{i}_1B^n_i\, \middle|\, \cF_{n-1}\right)\right)
\leq \mu^\gamma_{(x,V)}\left(\sum_{i=1}^{n-1}b^{i}_1B^n_i+\mu^\gamma_{(x,V)}\left(B^n_0 \, \middle|\,
\cF_{n-1}\right)\right),
\end{equation}
while from strong Markov property and \eqref{eq:tildeM1} we get
\begin{align}
\mu^\gamma_{(x,V)}\left(B^n_0\, \middle|\, \cF_{n-1}\right) & = \mu_{(x,V)}^{\gamma
}\left(b_2(X_{(n-1)\delta},X^{-}_{n\delta}-X_{(n-1)\delta})\, \middle|\, \cF_{n-1}\right)\nonumber\\
& =\mu_{X_{(n-1)\delta}}^{\gamma}\left(b_2(X_0,X_{\delta}-X_0)\right)\nonumber\\
& \leq \sup_{x\in E}\mu_x^{\gamma}\left(b_2(x,X_{\delta}-x)\right)\nonumber\\
& \leq \sup_{x\in E}\mu_x^{|\gamma|}\left(b_2(x,X_{\delta}-x)\right)=\tilde M_2(|\gamma|).\label{eq:pr3}
\end{align}
Consequently, combining \eqref{eq:pr3}, \eqref{eq:pr2}, and \eqref{eq:pr1} we get
\begin{align*}
\mu^\gamma_{(x,V)}(\omega(X_{n\delta})) & \leq \omega(x)+a+\tilde
M_2(|\gamma|)+\mu^\gamma_{(x,V)}\left(\sum_{i=1}^{n-1}\1_{A_{n-i}}b^{i}_1 B^n_i\right).
\end{align*}
Using similar reasoning recursively and noting that for $i=1,\ldots,n-1$ we have
\begin{align}
\mu^{\gamma}_{x,V}\left(b^{i}_1(\delta)B^n_i\, \middle|\, \cF_{n-i-1}\right) & \leq \sup_{x\in
E}\mu_x^{\gamma}\left(b^{i}_1(\delta)b_2(x,X_{\delta}-x)\right)\nonumber\\
& =  b^{i}_1(\delta)\sup_{x\in E}\mu_x^{\gamma b^{i}_1(\delta)}\left(b_2(x,X_{\delta}-x)\right)\nonumber\\
& \leq  b^{i}_1(\delta)\sup_{x\in E}\mu_x^{|\gamma|}\left(b_2(x,X_{\delta}-x)\right)\nonumber\\
& =  b^{i}_1(\delta)\tilde M_2(|\gamma|),\label{eq:omega.n4}
\end{align}
we finally get
\begin{equation}\label{eq:omega.n2}
\mu^\gamma_{(x,V)}(\omega(X_{n\delta}))\leq \omega(x)+a+\tilde M_2(|\gamma|)\sum_{i=0}^{n-1}b_1^{i} \leq
\omega(x)+a+\tfrac{1}{1-b_1}\tilde M_2(|\gamma|).
\end{equation}
As the choice of $V\in\bV_{\delta}$ and $t\in\bT_{\delta}$ was arbitrary, and the upper bound in \eqref{eq:omega.n2} is
independent of both, we know that \eqref{eq:w.omega} is satisfied on $\bT_\delta$ which concludes the proof.
\end{proof}

Finally, we are ready to link Bellman's equation to the corresponding dyadic optimal control problem~\eqref{eq:main1}.

\begin{proposition}\label{pr:RSC.Bellman.solution2}
Under assumptions \eqref{A.1}--\eqref{A.4} we get
\[
\lambda^{\gamma}_{\delta}/\delta=\sup_{V\in\bV_{\delta}}J_\gamma(x,V),
\]
i.e. the optimal value in problem \eqref{eq:main1} corresponds to the solution of Bellman equation
\eqref{eq:rsc:bellmaneq}.
\end{proposition}

\begin{proof}
Proposition~\ref{pr:w.omega}.
For brevity and with slight abuse of notation, for any $n\in\bN$ we set $T_n:=n\delta$ and
\[
Z_n:={\int_0^{T_n} f(X_s)\d s +\sum_{i=1}^{\infty} \1_{\{ \tau_i\leq {T_n}\}}c(X_{\tau^-_i},\xi_i)};
\]
note the exact dynamics of $Z_n$ is determined by an underlying strategy $V=\{(\tau_i,\xi_i)\}_{i=1}^{\infty}$.

First, let us show that
\begin{equation}\label{eq:ver1}
\lambda^{\gamma}_{\delta}/\delta\leq \sup_{V\in\bV_{\delta}}J_\gamma(x,V).
\end{equation}
Fix $n\in\bN$, $p>1$, and set $\bar \gamma :=p\gamma$. Let $q$ be the conjugate index for $p$ and let
$\phi:=-\|w^{\bar\gamma}_{\delta}\|_{\omega}q\gamma$. For the strategy $\hat V=\{(\hat \tau_i,\hat
\xi_i)\}_{i=1}^{\infty}\in \bV_{\delta}$ determined by the Bellman equation~\eqref{eq:rsc:bellmaneq} for $\bar\gamma$,
using reverse H{\"o}lder's inequality for $p$ and $q$ (see Lemma~\ref{lm:holder}), we get
\begin{align}
\lambda^{\bar\gamma}_{\delta}/\delta & =\frac{1}{T_n}\mu^{\bar\gamma}_{(x, \hat
V)}\left(Z_n+w^{\bar\gamma}_{\delta}(X_{T_n})-w^{\bar\gamma}_{\delta}(x)\right)\nonumber\\
 & \leq \frac{1}{T_n}\left[\mu^{\bar\gamma}_{(x, \hat
 V)}\left(Z_n+\|w^{\bar\gamma}_{\delta}\|_{\omega}\omega(X_{T_n})\right)+\|w^{\bar\gamma}_{\delta}\|_{\omega}
 -w^{\bar\gamma}_{\delta}(x)\right]\nonumber\\
 & \leq \frac{1}{T_n}\left[\mu_{(x, \hat V)}^{\bar\gamma/p}\left(Z_n\right) +\mu_{(x, \hat
 V)}^{-q\bar\gamma/p}\left(\|w^{\bar\gamma}_{\delta}\|_{\omega}\omega(X_{T_n})\right)+\|w^{\bar\gamma}_{\delta}\|_{\omega}
 -w^{\bar\gamma}_{\delta}(x)\right]\nonumber\\
  & \leq \frac{1}{T_n}\left[\mu_{(x, \hat V)}^{\gamma}\left(Z_n\right)
  +\|w^{\bar\gamma}_{\delta}\|_{\omega}\mu_{(x,\hat
  V)}^{\phi}\left(\omega(X_{T_n})\right)+\|w^{\bar\gamma}_{\delta}\|_{\omega}
  -w^{\bar\gamma}_{\delta}(x_0)\right].\label{eq:lambda2}
\end{align}
Using Proposition~\ref{pr:w.omega} we know that $\sup_{n\in\bN}\mu_{(x,V)}^{\phi}\left(\omega(X_{T_n})\right)<\infty$.
Consequently, letting $n\to \infty$ we obtain
\begin{align}
\lambda_{\delta}^{\bar\gamma}/\delta & \leq \liminf_{n\to\infty}\frac{1}{T_n}\left[\mu_{(x,\hat
V)}^{\gamma}\left(Z_n\right)
+\|w^{\bar\gamma}_{\delta}\|_{\omega}\sup_{n\in\bN}\mu_{(x,V)}^{\phi}\left(\omega(X_{T_n})\right)+\|w^{\bar\gamma}_{\delta}\|_{\omega}
-w^{\bar\gamma}_{\delta}(x_0)\right]\nonumber\\
&=   \liminf_{n\to\infty}\frac{1}{T_n}\mu_{(x,\hat V)}^{\gamma}\left(Z_n\right)\leq
\sup_{V\in\bV_{\delta}}J_\gamma(x,V).\label{eq:ver4}
\end{align}
Now, recall that $\bar\gamma=p\gamma$ and note that \eqref{eq:ver4} holds for any choice of $p>1$. Thus, using
Proposition~\ref{pr:lambda.continuity} and letting $p\to 1$, we get that  $\lambda_{\delta}^{p\gamma} \to
\lambda_{\delta}^{\gamma}$. This concludes the proof of \eqref{eq:ver1}.

Second, we prove inequality
\begin{equation}\label{eq:ver2}
\lambda^{\gamma}_{\delta}/\delta\geq \sup_{V\in\bV_{\delta}}J^{\delta}_\gamma(x,V).
\end{equation}
Again, we fix $n\in\bN$ and $p>1$. Let $\bar \gamma :=\gamma/p$ and
$\phi:=-\|w^{\bar\gamma}_{\delta}\|_{\omega}q\bar\gamma$, where $q$ is the conjugate index for $p$. For any strategy
$V\in\bV_{\delta}$, using H{\"o}lder's inequality for $p$ and $q$ (see Lemma~\ref{lm:holder}), we get
\begin{align}
\lambda^{\bar\gamma}_{\delta}/\delta & \geq\frac{1}{T_n}\mu^{\bar\gamma}_{(x,
V)}\left(Z_n+w^{\bar\gamma}_{\delta}(X_{T_n})-w^{\bar\gamma}_{\delta}(x)\right)\nonumber\\
 & \geq \frac{1}{T_n}\left[\mu^{\bar\gamma}_{(x,
 V)}\left(Z_n-\|w^{\bar\gamma}_{\delta}\|_{\omega}\omega(X_{T_n})\right)-\|w^{\bar\gamma}_{\delta}\|_{\omega}
 -w^{\bar\gamma}_{\delta}(x)\right]\nonumber\\
 & \geq \frac{1}{T_n}\left[\mu_{(x, V)}^{p\bar\gamma}\left(Z_n\right) +\mu_{(x,
 V)}^{q\bar\gamma}\left(-\|w^{\bar\gamma}_{\delta}\|_{\omega}\omega(X_{T_n})\right)-\|w^{\bar\gamma}_{\delta}\|_{\omega}
 -w^{\bar\gamma}_{\delta}(x)\right]\nonumber\\
  & \geq \frac{1}{T_n}\left[\mu_{(x,V)}^{\gamma}\left(Z_n\right) -\|w^{\bar\gamma}_{\delta}\|_{\omega}\mu_{(x,
  V)}^{\phi}\left(\omega(X_{T_n})\right)-\|w^{\bar\gamma}_{\delta}\|_{\omega}
  -w^{\bar\gamma}_{\delta}(x_0)\right].\label{eq:lambda2}
\end{align}
As before, using Proposition~\ref{pr:w.omega} and letting $n\to \infty$, for any $V\in\bV_{\delta}$ we obtain
\[
\lambda_{\delta}^{\bar\gamma}/\delta\geq  \liminf_{n\to\infty}\frac{1}{T_n}\mu_{(x,V)}^{\gamma}\left(Z_n\right).
\]
As the choice of $V\in\bV$ is arbitrary we get
\[
\lambda_{\delta}^{\bar\gamma}/\delta \geq \sup_{V\in\bV_{\delta}}J^{\delta}_\gamma(x,V).
\]
Finally, as in the proof of \eqref{eq:ver1}, using Proposition~\ref{pr:lambda.continuity} and letting $p\to 1$, we get
$\lambda_{\delta}^{\gamma/p} \to \lambda_{\delta}^{\gamma}$, which concludes the proof of \eqref{eq:ver2}, and
Proposition~\ref{pr:RSC.Bellman.solution2}.

\end{proof}

\begin{remark}[Application of entropic H{\"o}lder's inequalities]
The key step in the proof of Proposition \ref{pr:RSC.Bellman.solution2} is the application on the Holder's inequality
and reverse Holder's inequality for the entropic risk; see Lemma~\ref{lm:holder}. Using the induced superadditivity and
subadditivity property (for different risk averse parameters), one can split the main dynamics from
$w^{\gamma}_{\delta}(\cdot)$. It is interesting to note that the same approach could be applied in \cite[Proposition
5]{PitSte2016}, i.e. using our framework it is easy to show that the solution to the Bellman's equation is the optimal
solution, without imposing any additional constraints as in \cite[Proposition 5]{PitSte2016}.
\end{remark}

\begin{remark}[Full time-grid]\label{rem:full.time}
While in Proposition~\ref{pr:w.omega} and Proposition~\ref{pr:RSC.Bellman.solution2} we restricted ourselves to the
dyadic time-grid, the results holds (under additional mild assumptions) on the full-time grid, i.e. with objective
function \eqref{eq:objective.function} replaced by
\[
\tilde J(x,V):=\liminf_{T\to\infty}\frac{J_{T}(x,V)}{T}.
\]
Following comments from
Remark~\ref{rem:time.grid1} and treating $b_1$ and $M_i$ in \eqref{A.3} as functions of $\delta$, let us assume that
$M_i(\gamma,\delta)\to 0$ as $\delta\to 0$, for any $\gamma\in\bR$. For brevity, let us only outline how to extend the
proof of Proposition~\ref{pr:w.omega}. Let $t>0$ be such that $t\not\in \bT_{\delta}$ and let
$V\in\bV_{\delta}$. We know that there exists $\delta_0<\delta$ such that $M:=\sup_{\delta\in
(0,\delta_0]}M_i(|\gamma|,\delta)<\infty$. Also, we know that there exist $n\in\bN$ and $m\in\bN$ such that $t=n\delta
+m\delta_0+\epsilon$, where $m\delta_0<\delta$ and $\epsilon\in [0,\delta_0)$. For brevity we set $t_0
:=n\delta+\epsilon$. Using \eqref{A.3} $m$-times for time step $\delta_0$ and once for time step $\epsilon$ (if
required), and using notation introduced in \eqref{eq:b2}, we get
\[
\omega(X_t)\leq \omega(X_{n\delta})+b_2(X_{n\delta},X^{-}_{t_0}-X_{n\delta})+ \sum_{i=0}^{m-1}b_1^i(\delta_0)b_2(X_{t_0
+(m-i-1)\delta_0},X_{t_0 +(m-i)\delta_0}-X_{t_0 +(m-i-1)\delta_0}).
\]
Now, using similar arguments as in the proof of \eqref{eq:omega.n2}, we get
\begin{align}
\mu^\gamma_{(x,V)}(\omega(X_{t})) & \leq \mu^\gamma_{(x,V)}(\omega(X_{n\delta})) +\tilde
M+\tfrac{1}{1-b_1(\delta_0)}\tilde M\nonumber\\
&\leq \omega(x)+a+ \tfrac{1}{1-b_1(\delta)}\tilde M_2(|\gamma|,\delta)+\tilde M+\tfrac{1}{1-b_1(\delta_0)}\tilde
M,\label{eq:upper.1.no}
\end{align}
where $\tilde M$ and $\tilde M_2(|\gamma|,\delta)$ is constructed as in \eqref{eq:tildeM1}.
As the choice of $\delta_0$ was independent of the choice of $t$ and $V$, so is the upper bound in
\eqref{eq:upper.1.no}. This concludes the proof of \eqref{eq:w.omega} for $t\in\bT$.

\end{remark}

\section{Reference examples}\label{S:examples}
In this section we want to show examples of processes satisfying assumptions \eqref{A.1}--\eqref{A.4}. For brevity, as assumptions \eqref{A.1}, \eqref{A.2}, and \eqref{A.4} are rather  standard, we decided to focus on assumption \eqref{A.3} and describe only the dynamic of the uncontrolled process; one could easily enhance this process to get a proper example satisfying \eqref{A.1}--\eqref{A.4}. Example~\ref{ex:1} focus on Ito-like diffusion process, Example~\ref{ex:2} is linked to regular step processes studied in \cite{BluGet2007}, and Example~\ref{ex:3} considers a piecewise deterministic process introduced in \cite{Dav1984} and studied later in the context of control theory in \cite{BauRie2011}. For simplicity, in the first two examples we assume that $E=\bR^d$ and $\delta<1$, and in the third we set $E=\bR$.

\begin{example}[Ito-like diffusion]\label{ex:1}
Let $(X_t)$ be a solution to equation
\begin{equation}\label{eqq}
 dX_t=(AX_t+g(X_t))dt + \sigma(X_t)dW_t,
\end{equation}
where matrix $A\in \bR^{d\times d}$ is stable (real parts of its eigenvalues are negative) and diagonalizable (its
geometric and algebraic multiplicities coincides), functions $g:\bR^d \to \bR^d$ and $\sigma: \bR^d \to \bR^{d\times d}$ are bounded, and $(W_t)$ is $\bR^d$-valued Brownian motion. Additionally, we assume
that $\sigma$ is Lipschitz continuous to guarantee strong solution of \eqref{eqq} with $g\equiv 0$. Then, there exists a
weak solution to \eqref{eqq} given by
\begin{equation}\label{eqqq}
X_t=e^{At}X_0 + \int_0^t e^{A(t-s)}g(X_s)ds + \int_0^t e^{A(t-s)}\sigma(X_s) dW_s.
\end{equation}
Let $\omega(x):=\max_{i\in \{1,\ldots,d\}} |x_i|$ for $x\in\bR^d$. Then, for any $t\leq 1$  and $\gamma\in \bR$ we get
\begin{equation}\label{eqqqq}
\mu_x^\gamma(\omega(X_t))\leq  e^{-\alpha t}\omega(x)+\|g\|_{\infty}+\mu_x^\gamma \left(\omega\left( \int_0^t e^{A(t-s)}\sigma(X_s) dW_s\right)\right),
\end{equation}
 where $\alpha\in \bR_{+}$ is a (negative of) maximal real  part of eigenvalues of $A$ and $\|\cdot\|_{\infty}$ denotes the supremum norm. We now show that the last term in \eqref{eqqqq} could be uniformly bounded for any $\gamma\in \bR$. For simplicity, and without loss of generality, we assume that $\gamma>0$; recall that entropic risk measure is monotone with respect to $\gamma$. Let
 \begin{equation}\label{eq:example.z}
 Z(t):=\int_0^t e^{A(t-s)}\sigma(X_s) dW_s,\quad\quad t\in \bR_{+},
\end{equation}
and let $Z_i(t)$ denote the i-th component of $Z(t)$, for $i=1,\ldots,d$. Notice that for any $\gamma> 0$ and $x\in E$ we have
 \begin{equation}\label{eq:example.z2}
 \bE_{x}\left[e^{\gamma \omega(Z(t))}\right]\leq \bE_{x}\left[e^{\gamma \sum_{i=1}^{d}|Z_i(t)|}\right] \leq \sum_{(s_1,\ldots,s_d)\in \left\{0,1\right\}^d} \bE_{x}\left[e^{\gamma \sum_{i=1}^{d} (-1)^{s_i} Z_i(t)}\right].
 \end{equation}
 By the local martingale property of $e^{\gamma \sum_{i=1}^{d} (-1)^{s_i}Z_i(t) - {1\over 2}\gamma^2 \langle\sum_{i=1}^{d}
 (-1)^{s_i}Z_i\rangle_t}$ (we refer to Problem 3.38 in \cite{KarShr1998b}), for any $x\in E$ we get
 \begin{equation}\label{eq5}
 \bE_{x}\left[e^{\gamma \omega(Z(t))}\right]\leq 2^d e^{{1\over 2}\gamma^2 d\|\sigma\|_{\infty}^2}.
\end{equation}
This completes the proof the second estimate in \eqref{A.3}. The first estimate in \eqref{A.3}, i.e. inequality $\mu_x^\gamma (\int_0^t \omega(X_s)ds)\leq \omega(x) + M_1(\gamma)$, can be obtained in a similar way by exploiting property \eqref{eqqq}.
\end{example}
\begin{example}[Regular step process]\label{ex:2}
Let $(X_t)$ be a regular step process that is constructed using the following logic: a particle is starting from point $X_0=z_0$ and remains in there for exponentially distributed time with parameter $r(z_0)$. Then, it jumps to another (randomly chosen) state $z_1$ and remains there for exponentially distributed time with parameter $r(z_1)$, and so on. The intensity function $r\colon \bR^{d} \to \bR_+$ is given by $r(\cdot):=\max \left\{ \|\cdot\|^{1+\epsilon},r_0 \right\}$, where $r_0>0$ and $\epsilon>0$ are fixed constants, and where $\|\cdot\|$ is the standard $\bR^d$-norm. The jump from $z_n\in \bR^d$ to $z_{n+1}\in \bR^d$ (for $n\in\bN$) is made according to the transition law $Q(z_n,\cdot)$ such that
\begin{equation}
z_{n+1}=A(z_n)+w_n,
\end{equation}
where  $(w_n)$  is an i.i.d. sequence of bounded $\bR^d$-valued random vectors, and the function $A\colon \bR^d\to \bR^d$ satisfy $\lim_{\|x\|\to \infty} \tfrac{1}{\|x\|}\|A(x)\|<1$. Then, there exists a constant $K\in \bR_{+}$ and $\beta \in (0,1)$ such that $\|A(x)\|+\|w_i\|\leq \beta \|x\| +K$ for $x\in E$ and $i\in\bN$. Consequently, for any $n\in\bN$, we get $\|z_{n+1}\|\leq \beta \|z_n\|+K$ and, by iteration,
\begin{equation}\label{eq:ex2.1}
\|z_{n+1}\|\leq \beta^{n+1}\|z_0\| + {K\over 1-\beta}.
\end{equation}
Let $\omega(\cdot):=\|\cdot\|$ and let $\tau(x)$ denote the time of the first process jump for any fixed starting point $x\in E$. Then, for any $x\in E$ and $\gamma>0$, using \eqref{eq:ex2.1}, we get
\begin{align}\label{eq:ex2.2}
\bE_x\left[e^{\gamma \omega(X_\delta)}\right] &\leq \bE_x\left[\1_{\{\delta<\tau(x)\}}e^{\gamma \omega(x)}\right]+\bE_x\left[\1_{\{\delta\geq \tau(x)\}} \sup_{n\in\bN}e^{\gamma\omega(z_n)}\right] \nonumber \\
&\leq \bP_{x}[\delta<\tau(x)]\cdot e^{\gamma \omega(x)}+ \bE_x\left[\1_{\{\delta\geq \tau(x)\}}\right]e^{\gamma\beta\omega(x)+\gamma\frac{K}{1-\beta}}\nonumber \\
&\leq e^{-\delta r(x)+\gamma \omega(x)}+ e^{\gamma \beta \omega(x)+\gamma{K\over 1-\beta}}.
\end{align}
Now, set $R(\gamma):=\max\{\sqrt[\epsilon]{\gamma/\delta},r_0\}$ and recall that $C_{R(\gamma)}=\{x\in \bR^d: \omega(x)<R(\gamma)\}$. For any $x\in C_{R(\gamma)}$, we get $e^{-\delta r(x)+\gamma \omega(x)}\leq e^{\gamma R(\gamma)}$, while for $x\not\in C_R$ we have
\begin{equation}
e^{-\delta r(x)+\gamma \omega(x)}\leq e^{-\delta\omega^{1+\epsilon}(x)}e^{\gamma \omega(x)}= e^{-\omega(x)(\delta R^\epsilon(\gamma)-\gamma)}\leq 1.
\end{equation}
Consequently, from \eqref{eq:ex2.2}, we get
\begin{equation}\label{eq:ex2.333}
\bE_x\left[e^{\gamma \omega(X_\delta)}\right]\leq \max\{e^{\gamma R(\gamma)}, 1\}+e^{\gamma \beta \omega(x)+\gamma{K\over 1-\beta}}\leq e^{\gamma \beta \omega(x) + \gamma \tilde K(\gamma)},
\end{equation}
where $\tilde K(\gamma)$ is some fixed constant independent of $x\in E$. This completes the proof of the second inequality in \eqref{A.3}, as \eqref{eq:ex2.333} could be rewritten as $\mu^{\gamma}_{x}(\omega(X_\delta)) \leq \beta \omega(x) + \tilde K(\gamma)$, for $x\in E$.
The proof of the first inequality in \eqref{A.3} follow directly from \eqref{eq:ex2.1}.

\end{example}

\begin{example}[Piecewise deterministic process]\label{ex:3} 
Assume that $(X_t)$ is a piecewise deterministic process. The deterministic part is a solution to a stable differential equation 
\begin{equation}\label{eqq1}
dX_t=F(X_t)dt,
\end{equation}
with initial state $X_0=x$. The process follows this dynamics till (random) jump moment, and then is subject to immediate shift after which its evolution follows the same deterministic logic till next jump occurs, and so on. We assume that the sequence of jumps, say $(\tau_n)$, is such that $(\tau_{n+1}-\tau_n)$ is i.i.d. and exponentially distributed with fixed intensity $r>0$. The shifts are made according to transition measure such that
\[
X_{\tau_n}=A(X_{\tau_n^-})+w_n,
\]
where $w_n$ is a sequence of i.i.d. standard normal random variables and  function $A\colon \bR\to \bR$ satisfy $|A(x)|\leq |x|+K$, for  $K>0$. Assuming suitable regularity of $F$, for any $t<\tau_1$ and initial state $x$, we get $X_t=\phi(x,t)$, where $\phi$ is a continuous function. Moreover, we assume that $\phi$ is such that for any $x\in E$ we get $|\phi(x,t)|\leq e^{-\alpha t}|x|+M$, where $\alpha, M>0$ are some predefined constants that are independent of $x$. Then, we get
\begin{equation}
\1_{\{\tau_1>\delta\}}|X_\delta| \leq \1_{\{\tau_1>\delta\}}\left(e^{-\alpha \delta}|x|+M\right),
\end{equation}
and, for any $n\in\bN$, by induction, 
\[
\1_{\{\tau_{n+1}>\delta\geq \tau_{n}\}}|X_\delta|\leq \1_{\{\tau_{n+1}>\delta\geq \tau_{n}\}}\left(e^{-\alpha\delta} |x|+ M+n(K+M) + \sum_{i=1}^{n} |w_i|\right).
\]
Consequently, for any $\gamma>0$, setting $\omega(\cdot):=\|\cdot\|$, $\beta:=e^{-\alpha\delta}$, $\tau_0:=0$, $w_0:=0$, and $D(\gamma):=\tfrac{1}{\gamma}\ln\bE[e^{\gamma |w_1|}]$, and noting that $(w_i)$ is independent of $(\tau_i)$, we get

\begin{align}
\bE_x\left[e^{\gamma \omega(X_\delta)}\right] &= \bE_x\left[\sum_{n=0}^{\infty} \1_{\{\tau_{n+1}>\delta\geq \tau_{n}\}}e^{\gamma \omega(X_\delta)}\right]\nonumber\\
&\leq e^{\gamma[\beta|x|+M]}\cdot\bE_x\left[\sum_{n=0}^{\infty} \1_{\{\tau_{n+1}>\delta\geq \tau_{n}\}}e^{\gamma \left[n(K+M)+\sum_{i=0}^{n}|w_i|\right]}\right]\nonumber\\
&\leq e^{\gamma[\beta|x|+M]}\cdot\sum_{n=0}^{\infty}\bE_x\left[\1_{\{\tau_{n+1}>\delta\geq \tau_{n}\}}\right]e^{n\gamma \left[K+M+D(\gamma)\right]}\nonumber \\
&\leq e^{\gamma[\beta|x|+M]}\cdot\sum_{n=0}^{\infty} \frac{(r\delta)^n\,e^{-r\delta}}{n!} \cdot e^{n\gamma \left[K+M+D(\gamma)\right]}.\label{eq:ex3.last}
\end{align}
Next, noting that 
\[
\sum_{n=0}^{\infty} \frac{(r\delta)^n\,e^{-r\delta}}{n!} \cdot e^{n\gamma \left[K+M+D(\gamma)\right]}<\infty,
\]
we can rewrite \eqref{eq:ex3.last} as
\[
\mu_{x}^{\gamma}(\omega(X_{\delta}))\leq \beta \omega(x)+\tilde D(\gamma),
\]
where $\tilde D(\gamma)$ is some constant that is independent of $x$; this concludes the proof of the left inequality in \eqref{A.3}. The second inequality in \eqref{A.3} follows in a similar manner.
\end{example}




\section{Appendix}\label{S:appendix}
For simplicity, in this section we assume that a probability space is fixed and for any $\gamma\in\bR\setminus\{0\}$ and $X\in L^0$ we set
\[
\mu^\gamma(X):=1/\gamma\, \ln \bE\left[\exp(\gamma X)\right].
\]
\begin{lemma}[H{\"o}lder's inequalities for entropic utility measure]\label{lm:holder}
Let $\gamma<0$ (resp. $\gamma>0$). Then, for any $p>1$ and the corresponding conjugate index $q$ we get
\begin{align}
\mu^{\gamma}(X+Y) &\geq \mu^{p\gamma}(X)+\mu^{q\gamma}(Y),\qquad \textrm{(resp. $\leq$)}\label{eq:holder1a}\\
\mu^{\gamma}(X+Y) & \leq \mu^{\gamma/p}(X)+\mu^{-q\gamma/p}(Y),\qquad \textrm{(resp. $\geq$)}\label{eq:holder1b}
\end{align}
where $X,Y\in L^0$.
\end{lemma}
\begin{proof}
We only show proof for $\gamma<0$ as the proof for $\gamma>0$ is analogous. Let us fix $p>1$. Using H{\"o}lder's
inequality applied to $e^{\gamma X}$ and $e^{\gamma Y}$ we get
\[
\bE\left[\exp(\gamma (X+Y))\right] \leq \bE[\exp(p\gamma X)]^{1/p}\bE[\exp(q\gamma Y)]^{1/q},
\]
Taking logarithm on both sides and multiplying by $1/\gamma<0$ we get
\[
\tfrac{1}{\gamma}\ln\bE\left[\exp(\gamma (X+Y))\right] \geq \tfrac{1}{p\gamma}\ln\bE[\exp(p\gamma
X)]+\tfrac{1}{q\gamma}\ln\bE[\exp(q\gamma Y)],
\]
which is equivalent to \eqref{eq:holder1a}. Next, applying \eqref{eq:holder1a} to $\tilde\gamma=\gamma/p$, $\tilde X:=
X+Y$, and $\tilde Y:=-Y$, we get
\[
\mu^{\gamma/p}(X) \geq \mu^{\gamma}(X+Y)+\mu^{q\gamma/p}(-Y)= \mu^{\gamma}(X+Y)-\mu^{-q\gamma/p}(Y),
\]
from which \eqref{eq:holder1b} follows.
\end{proof}

 {\small
 \bibliographystyle{agsm}

 }


 \end{document}